\documentclass[11pt]{amsart}
\usepackage{amssymb,latexsym,amsmath,amsthm,amsfonts, enumerate}
\usepackage{float}
\usepackage[colorlinks, linkcolor=blue!90, anchorcolor=Periwinkle,
    citecolor=blue!90,urlcolor=cyan, bookmarksopen,bookmarksdepth=2]{hyperref}
\usepackage[usenames,dvipsnames]{xcolor}
\usepackage{enumitem}
\usepackage{geometry} 
\usepackage{graphicx}
\usepackage{bookmark}
\usepackage{subfigure}
\usepackage{tikz}\usetikzlibrary{matrix}
\usepackage{url}

\input xy \xyoption{all}

\theoremstyle{plain}
\newtheorem{theorem}{Theorem}[section]
\newtheorem{lemma}[theorem]{Lemma}
\newtheorem{corollary}[theorem]{Corollary}
\newtheorem{proposition}[theorem]{Proposition}

\theoremstyle{definition}
\newtheorem{definition}[theorem]{Definition}
\newtheorem{example}[theorem]{Example}
\newtheorem{remark}[theorem]{Remark}

\numberwithin{equation}{section}

\def\hua{\mathcal}
\def\kong{\mathbb}
\def\<{\langle}
\def\>{\rangle}

\def\Aut{\operatorname{Aut}}
\def\Ind{\operatorname{Ind}}
\def\Sim{\operatorname{Sim}}
\def\Hom{\operatorname{Hom}}
\def\End{\operatorname{End}}
\def\Ext{\operatorname{Ext}}

\def\diff{\operatorname{d}}
\def\Br{\operatorname{Br}}

\def\Re{\operatorname{Re}}

\newcommand{\h}{\operatorname{\hua{H}}}            

\renewcommand{\k}{\mathbf{k}}
\renewcommand{\mod}{\operatorname{mod}}

\newcommand{\tilt}[3]{{#1}^{#2}_{#3}}
\newcommand{\Cone}{\operatorname{Cone}}

\def\K{\operatorname{K}}

\def\numbers
{\begin{enumerate}[label=\arabic*{$^\circ$}.]}
\def\ends
{\end{enumerate}}

\newcommand{\EG}{\operatorname{EG}}       

\newcommand{\D}{\operatorname{\hua{D}}}

\newcommand{\qq}{\Gamma Q}
\renewcommand{\ss}{\Gamma\kong{S}}
\newcommand{\cc}[1]{\overline{#1}}

\renewcommand{\i}{\mathbf{i}}
\renewcommand{\j}{\mathbf{j}}
\renewcommand{\a}{\mathbf{a}}
\renewcommand{\S}{\mathbf{\kong{S}}}
\newcommand{\id}{\operatorname{id}}

\newcommand{\F}{\kong{F}}

\newcommand{\cub}{\operatorname{U}}

\def\gq{\hua{G}_Q}
\def\gs{\hua{G}_\kong{S}}
\def\hq{\hua{H}_Q}
\def\hs{\h_{\S}}

\def\Stab{\operatorname{Stab}}
\newcommand\NS{\operatorname{NStab}}
\newcommand\FS{\operatorname{FStab}}
\newcommand\NSp{\NS^\circ}
\newcommand\FStap{\FS^\circ}
\def\Stap{\operatorname{Stab}^\circ}

\def\Frob{F_Q^\sigma}
\def\v{\varsigma}
\newcommand{\imm}{\hua{L}}
\newcommand{\hh}{\widetilde{\h}}
\def\gldim{\operatorname{gldim}}

\newcommand{\qihao}{\fontsize{7.25pt}{\baselineskip}\selectfont}
\newcommand{\bahao}{\fontsize{6.25pt}{\baselineskip}\selectfont}

\title{Frobenius morphisms and stability conditions}
\author{Wen Chang}
\thanks{This work is supported by Beijing Natural Science Foundation (Grant No. Z180003), National Key R\&D Program of China (NO. 2020YFA0713000) and by Natural Science Foundation of China (Grant No. 11601295).}
\address{Wen Chang:
School of Mathematics and Statistics,
Shaanxi Normal University,
Xi'an 710062, China.}
\email{changwen161@163.com}

\author{Yu Qiu}
\address{Yu Qiu:
	Yau Mathematical Sciences Center and Department of Mathematical Sciences,
	Tsinghua University,
    100084 Beijing,
    China.
    \&
    Beijing Institute of Mathematical Sciences and Applications, Yanqi Lake, Beijing, China}
\email{yu.qiu@bath.edu}

\UseRawInputEncoding
\begin{document}

\begin{abstract}
{\small

We generalize Deng-Du's folding argument,
for the bounded derived category $\D(Q)$ of an acyclic quiver $Q$,
to the finite dimensional derived category $\D(\qq)$ of the Ginzburg algebra $\qq$ associated to $Q$.
We show that the $F$-stable category of $\D(\qq)$ is equivalent to
the finite dimensional derived category $\D(\ss)$ of the Ginzburg algebra $\ss$
associated to the species $\S$, which is folded from $Q$.

If $(Q,\S)$ is of Dynkin type, we prove that
the space $\Stab\D(\S)$ of the stability conditions on $\D(\S)$ is canonically isomorphic to the space $\FS\D(Q)$ of $F$-stable stability conditions on $\D(Q)$. For the case of Ginzburg algebras, we also prove a similar isomorphism between principal components $\Stap\D(\ss)$ and $\FStap\D(\qq)$).

There are two applications.
One is for the space $\NS\D(\qq)$ of numerical stability conditions in $\Stap\D(\qq)$.
We show that $\NS\D(\qq)$ consists of $\Br Q/\Br \S$ many connected components, each of which is isomorphic to $\Stap\D(\ss)$,
for $(Q,\S)$ is of type $(A_3, B_2)$ or $(D_4, G_2)$.
The other is that we relate the $F$-stable stability conditions to the Gepner type stability conditions.

\vskip .3cm
{\parindent =0pt
\it Key words:} Folding; Frobenius morphism; Calabi-Yau category; Numerical stability condition; Gepner equation
}\end{abstract}

\maketitle

\setcounter{tocdepth}{1}

\tableofcontents

\section*{Introduction}
Bridgeland \cite{B1} introduced the notion of a stability condition on a triangulated category,
aiming to understand D-branes in string theory from a mathematical point of view.
We study the spaces of stability conditions arising from representation theory,
i.e. the categories associated to quivers and species.

One of our main techniques is folding, which is well-known in studying non-simply laced Dynkin diagrams.
In particular, folding the bounded derived category $\D(Q)$ of a quiver $Q$
was studied by Deng-Du \cite{DD1,DD2}.
The key observation is that an automorphism on the quiver $Q$ will induce a Frobenius morphism on the path algebra $\k Q$
and a Frobenius functor (which is an auto-equivalence) on the category $\D(Q)$.
Then the $F$-stable category of $\D(Q)$ is derived equivalent to the bounded derived category of a species $\S$,
which is obtained by folding the quiver $Q$.
Our first aim is to generalize this result
to the finite dimensional derived category $\D(\qq)$ of the Ginzburg algebra $\qq$ associated to $Q$, see Proposition~\ref{pp:Theta}.
Namely, the $F$-stable category of $\D(\qq)$ is derived equivalent to
the finite dimensional derived category $\D(\ss)$ of the Ginzburg algebra $\ss$ associated to $\S$.

Then applying the Frobenius functor on the stability conditions, we show in Theorem~\ref{thm:stable1} that, if $(Q,\S)$ is of Dynkin type, the space $\Stab\D(\S)$ of the stability conditions on $\D(\S)$
is canonically isomorphic to the space $\FS\D(Q)$ of $F$-stable stability conditions on $\D(Q)$, which is induced in Definition \ref{def:fstab stability conditions}.
For the case of derived categories of Ginzburg algebras, we have a similar isomorphism $\Stap\D(\ss)\cong\FStap\D(\qq)$ which is stated in Theorem~\ref{cor:stab CY}.

As an application of above results, we show in Theorem~\ref{thm:main} that, if $(Q,\S)$ is of type $(A_3, B_2)$ or $(D_4, G_2)$,
the space $\NS\D(\qq)$ of numerical stability conditions in the space $\Stap\D(\qq)$,
consists of $\Br\qq/\Br\ss$ many connected components,
each of which is isomorphic to $\Stap\D(\ss)$.
As another application, we see that the $F$-stable stability condition of $\D(Q)$ coincides with the stability condition of Gepner type $(F,0)$. When $\S$ is of Dynkin type, we explicitly compute the Gepner type stability condition of $\D(\S)$ with respect to the Auslander-Reiten translation, and show that the minimal value of the global dimension on $\Stab\D(\S)$ is $1-\frac{2}{h}$, where $h$ is the Coxeter number of $\S$.

The original version of this paper was motivated by a chat with Tom Sutherland and Alastair King.
Sutherland studies a list of quivers (known as Painlev\'{e} quivers) in his PhD work \cite{S},
whose corresponding spaces of numerical stability conditions are related to elliptic surface.
Note that all those quivers are foldable except one.
Moreover, the space of the numerical stability conditions for a quiver $Q$ is related to
the cluster algebra, whose type is the corresponding species $\S$ folded from $Q$ (cf. \cite{Q}).
The updated version is motivated by \cite{Q3} where we observe
the interaction between folding and Gepner equations (for stability conditions).

We should also mention that the idea that inducing stability conditions by some well-behaved functor has appeared in \cite{P07,MMS09}, where the authors focused on categories in certain geometric contexts. In the paper we focus on the categories arising from quiver representation theory and heavily use the folding technique on quivers and these categories.

\subsection{Notations and conventions}
Throughout, let $q$ be a prime power, let $\F_q$ be a finite field with $q$ elements and let $\k=\cc{\F_q}$ be the algebraic closure of $\F_q$. Let $\v=\v_q$ ($k\mapsto k^q$) be a field automorphism of $\k$ which is a power of Frobenius automorphism. We will use $\k'$ to denote the field $\k$ or $\F_q$.
We also assume that the categories we consider are all $\k'$-Hom-finite and Krull-Schmidt.

Here are some notations appearing in the paper.
\begin{itemize}
\item $Q=(Q_0,Q_1)$: quiver.
\item $\S=(\S_0,\S_1)$: species.
\item $\hq$: module category of $\k Q$.
\item $\hs$: module category of $\F_q \S$.
\item $\qq$: Ginzburg algebra of $Q$.
\item $\ss$: Ginzburg algebra of $\S$.
\item $\gq$: standard heart of $\D(\qq)$.
\item $\gs$: standard heart of $\D(\ss)$.
\item $\Sim\hua{A}$: the set of simple modules in an abelian category $\hua{A}$.
\item $\Br(\qq):$ the spherical twist group of $\D(\qq)$.
\item $\Br(\ss):$ the spherical twist group of $\D(\ss)$.
\item $\EG \D:$ the  total exchange graph of a triangulated category $\D$.
\item $\EG(\D, \h_0):$ the connected component of $\EG \D$ containing heart $\h_0$.
\item $\EG_3(\D, \h_0):$ the interval connected component of $\EG \D$ containing heart $\h_0$.
\item $\Stab \D:$ the space of stability conditions of a triangulated category $\D$.
\item $\NS\D:$ the space of numerical stability conditions.
\item $\FS\D:$ the space of $F$-stable stability conditions.
\end{itemize}

\subsection*{Acknowledgements}
Wen Chang would like to thank Dong Yang for answering the questions about the derived category and its heart via e-mail, and Xiangdong Yang for explaining some concepts about manifolds.
Qy would like to thank Tom Sutherland for sharing his ideas in his PhD thesis,
Alastair King for teaching him the folding technique,
Bernhard Keller for some explanations via e-mail
and Thomas Br\"{u}stle for proofreading.
Both authors thank reviewer's careful reading and many valuable suggestions.
Qy is supported by
National Key R\&D Program of China (No. 2020YFA0713000),
Beijing Natural Science Foundation (Grant No.Z180003) and
National Natural Science Foundation of China (Grant No.12031007).

\section{Preliminaries}\label{sec:pre}
\subsection{Folding from quivers to species}
\label{subsection:Folding from quivers to species}
Recall that a \emph{quiver} $Q=(Q_0,Q_1)$ is an oriented diagram, with (finite) vertex set $Q_0$ and (finite) arrow set $Q_1$. Denote by $h$ (resp. $t$) the map from $Q_1$ to $Q_0$ which maps an arrow to its head (resp. tail). An \emph{automorphism} $\iota$ of a quiver is a permutation of $Q_0$ such that for any arrow $a$, there is a unique arrow $\iota(a)$ with $h(\iota(a))=\iota(h(a))$ and $t(\iota(a))=\iota(t(a))$.
So $\iota$ is identity if and only if it is an identity on the vertex set $Q_0$. We also assume that $\iota$ is admissible, that is, there are no arrows connecting vertices in the same orbit of $\iota$ in $Q_0$.

Recall that an \emph{$\F_q$-species} $\S=(\S,L_\i,X_\a)$ consists of the following data:
\begin{itemize}
\item
    a quiver $\S=(\S_0,\S_1)$;
\item
    an $\F_q$-division ring $L_\i$ for each vertex $\i \in \S_0$;
\item
    an $L_\i$-$L_\j$-bimodule $X_\a$ for each arrow $\a\colon\i \to \j$ in $\S_1$.
\end{itemize}
Let $L=\oplus_{\i\in\S_0}L_\i$ and  $X=\oplus_{\a\in\S_1}X_\a$. Then $X$ is a natural $L$-$L$-bimodule. The $L$-algebra
$$\F_q \S:=\bigoplus_{n\ge 0}X^{\otimes n}\text{ where } X^{\otimes0}=L,
X^{\otimes n}=\underbrace{X\otimes_L\cdots\otimes_LX}_n$$ is
called the {\it path (or tensor) algebra} of $\S$. Thus,
a tensor $x_n\otimes\cdots\otimes x_1$ (write $x_n
\cdots x_1$ for simplicity) with  $x_i\in X_{\a_i}$
is non-zero implies that $\a_n\cdots \a_1$ is a path in $\S$.

Let $Q$ be an acyclic quiver, that is, there are no cyclic paths consisting of arrows on it.
For an automorphism $\iota$ of $Q$, one may fold it as a species. We recall this from \cite[Section 3, Section 6]{DD1}. Let $\k Q$ be the path algebra associated to $Q$ (note that a quiver is a special species), there is a Frobenius morphism $F=\Frob(q)$ on it given by
\begin{gather}\label{eq:Frob}
    F(\sum_{s} k_s p_s)=\sum_{s}\v(k_s) \iota(p_s),
\end{gather}
where $\sum_{s}k_s p_s$ is a $\k$-linear combination of paths in $\k Q$, and $\iota(p_s)=\iota(a_n)\cdots\iota(a_1)$ if $p_s=a_n\cdots a_1$ in $Q_1$.
Then one can use $F$ to define $\F_q$-species $\S=Q^\iota$ with $L_\i$ and $X_\a$ consisting of $F$-stable (see \cite{DD1} for precise meaning) objects. More precisely,
\begin{itemize}
\item the quiver $\S$ is the $\iota$-orbit of $Q$, i.e. $\S_0=Q_0/\iota$ and $\S_1=Q_1/\iota$;
\item for $\i\in\S_0$, denote by $|\i|$ the number of vertices in the orbit
    and fix $i_0\in\i$,
let
\begin{gather}
    L_\i=(\k Q)_\i^F=\{ \sum_{s=0}^{|\i|-1} \varsigma^s(k) {\iota^s(i_0)}
        \mid k\in \k, \v^{|\i|}(k)=k \},
\end{gather}
where ${\iota^s(i_0)}$ is the idempotent element corresponding to vertex $\iota^s(i_0)$ in $Q_0$ (we abuse notation here);
\item  for $\a\in\S_1$, similarly define $|\a|$ and fix $a_0\in\a$,
let
\begin{gather}
    X_\a=(\k Q)_\a^F=\{ \sum_{s=0}^{|\a|-1} \varsigma^s(k) \iota^s(a_0)
        \mid  k\in \k, \v^{|\a|}(k)=k  \}.
\end{gather}
\end{itemize}
Note that induced by the algebraic structure of $\k Q$, $L_{\i}$ is a $\F_q$-division ring of dimension $|\i|$ with $\sum_{s=0}^{|\i|-1}{\iota^s(i_0)}$ as identity, for any $\i$ in $\S_0$. More precisely, by viewing $\F_{q^{|\i|}}=\F_q(Y)\subseteq \k$ as a field extension over $\F_q$, $L_{\i}$ has a basis
\begin{gather}\label{equ:basis}
\{\sum_{s=0}^{|\i|-1}\varsigma^{s}(Y^j){\iota^s(i_0)}, 0 \leq j \leq |\i|-1\}.
\end{gather}
Similarly, $X_\a$ is a $L_\i$-$L_{\j}$-bimodule of $\F_q$-dimension $|\a|$, for any $\a\colon\mathbf{i} \to \j$ in $\S_1$, with a basis
\begin{gather}
\{\sum_{s=0}^{|\a|-1}\varsigma^{s}(Y^j)\iota^s(a_0), 0 \leq j \leq |\a|-1\}.
\end{gather}
Then $\S$ is an acyclic quiver and $\F_q \S$ is a finite dimensional $\F_q$-hereditary algebra which is isomorphic to the $F$-stable $\F_q$-subalgebra  $(\k Q)^F$ of $\k Q$ (\cite[Section 6]{DD1}).

\begin{example}
\label{ex:Dynkin}
When $Q$ is of Dynkin type,
all possible admissible automorphism $\iota$ and the corresponding species $\S$ are listed as follows (we omit the orientations but they should be compatible with $\iota$):
\numbers
\item
    $Q$ is of type $A_{2n-1}$ and $\S$ is of type $C_n$ while
    $\iota$ exchanges the bullets in the same column.
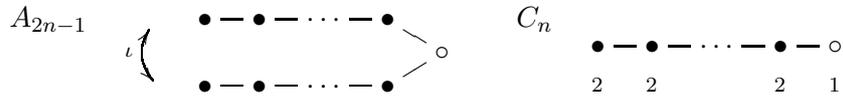
\[
    A_{2n-1} \quad
    \xymatrix@R=0.1pc@C=0.7pc{
        \ar@/_/@{<->}[dd]_{\iota} & \bullet  \ar@{-}[r]& \bullet \ar@{-}[r]&
            \cdots \ar@{-}[r]& \bullet \ar@{-}[dr]\\
        &&&&& \circ \\
        &\bullet \ar@{-}[r]& \bullet \ar@{-}[r]&
            \cdots \ar@{-}[r]& \bullet \ar@{-}[ur]
    }
    \qquad
    C_n \quad
    \xymatrix@R=0.1pc@C=0.7pc{
        \\
        \bullet \ar@{-}[r]& \bullet \ar@{-}[r]& \cdots \ar@{-}[r]& \bullet \ar@{-}[r]& \circ\\
        _2 & _2 && _2 & _1
    }
\]
\item
    $Q$ is of type $D_{n+1}$ and $\S$ is of type $B_n$ while
    $\iota$ exchanges the bullets.
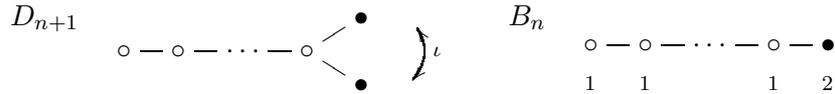
\[
    D_{n+1} \quad
    \xymatrix@R=0.1pc@C=0.7pc{
        &&&& \bullet &\ar@/^/@{<->}[dd]^{\iota}\\
        \circ \ar@{-}[r]& \circ \ar@{-}[r]& \cdots \ar@{-}[r]& \circ \ar@{-}[dr]\ar@{-}[ur]\\
        &&&& \bullet &\\
    }
    \qquad
    B_n \quad
    \xymatrix@R=0.1pc@C=0.7pc{
        \\
        \circ \ar@{-}[r]& \circ \ar@{-}[r]& \cdots \ar@{-}[r]& \circ \ar@{-}[r]& \bullet\\
        _1 & _1 && _1 & _2
    }
\]

\item
    $Q$ is of type $E_6$ and $\S$ is of type $F_4$ while
    $\iota$ exchanges the bullets in the same column.
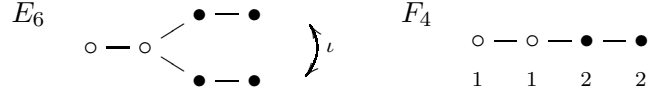
\[
    E_6 \quad
    \xymatrix@R=0.1pc@C=0.7pc{
        && \bullet \ar@{-}[r]& \bullet & \ar@/^/@{<->}[dd]^{\iota}\\
        \circ \ar@{-}[r]& \circ \ar@{-}[dr]\ar@{-}[ur]\\
        && \bullet \ar@{-}[r]& \bullet &\\
    }
    \qquad
    F_4 \quad
    \xymatrix@R=0.1pc@C=0.7pc{
        \\
        \circ \ar@{-}[r]& \circ \ar@{-}[r]& \bullet \ar@{-}[r]& \bullet\\
        _1 & _1 & _2 & _2
    }
\]
\item
    $Q$ is of type $D_4$ and $\S$ is of type $G_2$ while
    $\iota$ cyclically permutes the three bullets.
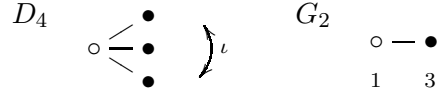
\[
    D_4 \quad
    \xymatrix@R=0.1pc@C=0.7pc{
        & \bullet & \ar@/^/@{<->}[dd]^{\iota}\\
        \circ \ar@{-}[r]\ar@{-}[dr]\ar@{-}[ur]& \bullet\\
        & \bullet &\\
    }
    \qquad
    G_2 \quad
    \xymatrix@R=0.1pc@C=0.7pc{
        \\
        \circ \ar@{-}[r]& \bullet\\
        _1 & _3
    }
\]
\ends
The label on vertex $\i$ in $\S$ are $|\i|$.
\end{example}

\subsection{Ginzburg algebras}
\label{subsection:Ginzburg algebras of quivers and species}
For a quiver $Q$, the \emph{Ginzburg (dg) algebra} (of degree 3) $\qq{}$ is constructed as follows (\cite[Section~7.2]{K10}):
\begin{itemize}
\item   Let $\overline{Q}$ be the graded quiver whose vertex set is $Q_0$
and whose arrows are: the arrows in $Q$ with degree $0$;
an arrow $a^*:j\to i$ with degree $-1$ for each arrow $a:i\to j$ in $Q$;
a loop $i^*:i\to i$ with degree $-2$ for each vertex $i$ in $Q$.
\item   The underlying graded algebra of $\qq{}$ is the completion of
the graded path algebra $\k \overline{Q}$ in the category of graded vector spaces
with respect to the ideal generated by the arrows of $\overline{Q}$.
\item   The differential of $\qq{}$ is the unique continuous linear endomorphism homogeneous
of degree $1$ which satisfies the Leibniz rule and
takes the following values on the arrows of $\overline{Q}$
\begin{gather}
   \diff a=0, \quad \diff (a^*)=0, \quad  \diff (i^*)=i(\sum_{a\in Q_1} \, [a,a^*])i .
\end{gather}
\end{itemize}

For a $\F_q$-species $\S=(\S,L_\i,X_\a)$ obtained by folding a quiver $Q$ in Section \ref{subsection:Folding from quivers to species}, we have the Ginzburg algebra $\ss$ constructed as follows
\begin{itemize}
\item   Let $\cc{\S}$ be the graded species whose vertex set is $\S_0$
(associated with the same division rings)
and whose arrows are: the arrows in $\S_1$ and the same bimodules with degree $0$;
an arrow $\a^*:\j\to \i$ and a $L_\j$-$L_\i$-bimodule $X_\a^*$ (dual over $\F_q$) with degree $-1$,
for each arrow $\a:\i\to \j$ in $\S_1$;
a loop $\i^*:\i\to \i$ with a $L_{\i}$-$L_{\i}$-bimodule
$L_\i^*$ (dual over $\F_q$) with degree $-2$ for each vertex $\i$ in $\S$.
\item   The underlying graded algebra of $\ss$ is the completion of
the graded path algebra $\k \cc{\S}$ in the category of graded vector spaces
with respect to the ideal generated by the arrows of $\cc{\S}$.
\item   The differential of $\ss$ is the unique continuous linear endomorphism homogeneous
of degree $1$ which satisfies the Leibniz rule and
takes the following values
\begin{gather}
    \diff (L_\i)=0, \quad \diff (X_\a)=0, \quad \diff (X_\a^*)=0, \quad  \diff (L_\i^*)=\i(\sum_{\a\in \S_1} \, [X_\a,X_\a^*])\i,
\end{gather}
where the value of the differential on a space is zero means that the value is zero on each element in the space. The last equality needs some explanation,  let $\i\in \S_0$ be a vertex with a representative $i_0$, then the differential on a basis element of $L_{\i}^*$ (see \eqref{equ:basis} for the basis of $L_{\i}$) is defined as
\begin{equation}\label{equ:differential}
\begin{array}{ll}
&d\left(\left(\sum\limits_{s=0}^{|\i|-1}\varsigma^{s}(Y^j){\iota^s(i_0)}\right)^*\right)\\
\quad \\
=&\left(\sum\limits_{s=0}^{|\i|-1}\iota^s(i_0)\right)
\left(\sum\limits_{\a\in \S_1}\left[\sum\limits_{s=0}^{|\a|-1}\varsigma^{s}(Y^j)\iota^{s}(a_0),
\left(\sum\limits_{s=0}^{|\a|-1}\varsigma^{s}(Y^j)\iota^{s}(a_0)\right)^*\right]\right)
\left(\sum\limits_{s=0}^{|\i|-1}\iota^s(i_0)\right)\\
\quad \\
=&\sum\limits_{\substack{\a\in \S_1 \\ t(a_0)=i_0}}\left(\sum\limits_{s=0}^{|\a|-1}\varsigma^{s}(Y^j)\iota^{s}(a_0)\right)\otimes
\left(\sum\limits_{s=0}^{|\a|-1}\varsigma^{s}(Y^j)\iota^{s}(a_0)\right)^*\\
&-\sum\limits_{\substack{\a\in \S_1 \\ h(a_0)=i_0}}
\left(\sum\limits_{s=0}^{|\a|-1}\varsigma^{s}(Y^j)\iota^{s}(a_0)\right)^*\otimes
\left(\sum\limits_{s=0}^{|\a|-1}\varsigma^{s}(Y^j)\iota^{s}(a_0)\right),
\end{array}
\end{equation}
where $a_0$ is a representative in $\a$ and $\sum_{s=0}^{|\i|-1}\iota^s(i_0)$ is the identity in $L_\i$ .

\end{itemize}

\begin{remark}\label{rem:Frobmorp of dg algebra}
Let $\overline{\iota}$ be an automorphism of $\overline{Q}$ induced from $\iota$ as follows
\[
\overline{\iota}(a)=\iota(a), \quad \overline{\iota}(a^*)=(\iota(a))^*, \quad \overline{\iota}(i^*)=(\iota(i))^*.
\]
Then $\overline{\iota}$ is an admissible automorphism of $\Gamma Q$. As in \cite[Section 6]{DD1},
we have a (homogeneous) Frobenius morphism $F$ on $\qq$ induced by $\overline{\iota}$
such that the restriction to $\k Q$ (i.e. the degree zero part of $\qq$) is exactly \eqref{eq:Frob}.
Similarly, one may consider the $\F_q$-algebra $(\Gamma Q)^F$ consisting of the $F$-stable elements in $\Gamma Q$, and an isomorphism $\Gamma \S\cong (\Gamma Q)^F$ of dg algebras. In particular, under such an isomorphism the differential \eqref{equ:differential} on $\Gamma \S$ coincides with the differential on $(\Gamma Q)^F$:
\begin{equation}\label{equ:differential2}
\begin{array}{ll}
&d\left(\sum\limits_{s=0}^{|\i|-1}\varsigma^{s}(Y^j)({\iota^s(i_0)})^*\right)\\
\quad \\
=&\left(\sum\limits_{s=0}^{|\i|-1}\iota^s(i_0)\right)
\left(\sum\limits_{\a\in \S_1}\left[\sum\limits_{s=0}^{|\a|-1}\varsigma^{s}(Y^j)\iota^{s}(a_0),
\sum\limits_{s=0}^{|\a|-1}\varsigma^{s}(Y^j)(\iota^{s}(a_0))^*\right]\right)
\left(\sum\limits_{s=0}^{|\i|-1}\iota^s(i_0)\right)\\
\quad \\
=&\sum\limits_{\substack{\a\in \S_1 \\ t(a_0)=i_0}}\left(\sum\limits_{s=0}^{|\a|-1}\varsigma^{s}(Y^j)\iota^{s}(a_0)\right)\otimes
\left(\sum\limits_{s=0}^{|\a|-1}\varsigma^{s}(Y^j)(\iota^{s}(a_0))^*\right)\\
&-\sum\limits_{\substack{\a\in \S_1 \\ h(a_0)=i_0}}
\left(\sum\limits_{s=0}^{|\a|-1}\varsigma^{s}(Y^j)(\iota^{s}(a_0))^*\right)\otimes
\left(\sum\limits_{s=0}^{|\a|-1}\varsigma^{s}(Y^j)\iota^{s}(a_0)\right).
\end{array}
\end{equation}
Note that in \cite{K8}, Keller considered the $n$-Calabi-Yau completion of any homologically smooth algebra, and showed that the Ginzburg algebra can be viewed as a (deformed) $3$-Calabi-Yau completion of the path algebra. Such a completion is compatible with the Frobenius map $F$ and the following diagram is commutative, where the vertical arrows are the $3$-Calabi-Yau completion.
\begin{gather}
\label{eq:diagramed-algebra}
    \xymatrix@C=3pc{
    \F_q\S \ar[r]^\cong\ar[d] & (\k Q)^F \ar@{^{(}->}[r] \ar[d] & \k Q \ar[d]\\
    \Gamma\S \ar[r]^{\cong} & (\qq)^F \ar@{^{(}->}[r] & \qq}
    \end{gather}
\end{remark}

\subsection{Derived categories}
\label{subsection:Finite dimensional derived categories}
From now on, for a species $\S$, we always assume it is obtained from folding a quiver $Q$.
Write $\D(\qq)$ for $\hua{D}_{fd}(\mod  \qq)$,
the \emph{finite dimensional derived category} of $\qq$
(cf. \cite[Section~7.3]{K10}).
Similarly, write $\D(\ss)$ for $\D_{fd}(\mod\ss)$,
the derived category of finite dimensional dg modules for $\ss$. Then \cite[Theorem 4.8]{K8}
says they are both Calabi-Yau-$3$ in the following sense.

Let $\k'$ be a field ($\k$ or $\F_q$ for the case of quiver or species respectively).
Recall that a $\k'$-linear triangulated category $\hua{D}$ is called \emph{Calabi-Yau}-$3$,
if for any objects $L,M$ in $\hua{D}$ we have non-degenerate bifunctorial pairings
\begin{gather}\label{eq:serre}
    \Hom^{i}(M,L)\times
    \Hom^{3-i}(L,M)\rightarrow \k'.
\end{gather}

Let $\D(Q)$ (resp. $\D(\S)$) be the bounded derived category of module category $\h_Q=\mod \k Q$ (resp. $\h_\S=\mod \F_q\S$). Then there is an exact and faithful functor (called \emph{L-immersion}, cf. \cite{K8}\cite[Section~7.3]{KQ})
\begin{equation}\label{map:L immersion-quiver}
    \imm_Q\colon\D(Q) \to \D(\qq)
\end{equation}
such that, for any pair of objects $(M,L)$ in $\D(Q)$,
there is a short exact sequence
\begin{equation}\label{eq:lagrangian}
    0 \to  \Hom^{\bullet}(M,L)
    \xrightarrow{\imm_Q}
    \Hom^{\bullet}(\imm_Q(M),\imm_Q(L))
    \xrightarrow{\imm_Q^\dagger}
    \Hom^{\bullet}(M,L)^\vee[-3] \to 0.
\end{equation}
Here, the dual of a graded vector space $V=\oplus_{i\in\kong{Z}} V_i[i]$ is
\[V^\vee=\bigoplus_{i\in\kong{Z}} V_i^*[-i],\]
where $V_i$ is an ungraded vector space and $V_i^*$ is its usual dual.

Similarly, we have an L-immersion
\begin{equation}\label{map:L immersion-species}
    \imm_{\S}\colon\colon\D(\S)\to\D(\ss)
\end{equation}
satisfying the corresponding \eqref{eq:lagrangian}.

\subsection{Tilting}
For a full subcategory $\hua{P}$ in a triangulated category $\D$, define
\[\hua{P}^{\perp}
=\{L\in \D:\Hom_{\D}(M,L)=0 \text{ for
all }M\in\hua{P}\}.\]
We call $\hua{P}$ a \emph{t-structure} on $\D$ if $\hua{P}[1]\subset\hua{P}$ and for every
object $N\in\D$ there is a triangle $M\to N\to L$ in $\D$
with $M\in\hua{P}$ and $L\in \hua{P}^{\perp}$.
Its \emph{heart} is defined to be
\[\hua{H}=\hua{P}\cap\hua{P}^{\perp}[1],\]
which is an abelian category.
We always assume that a t-structure is \emph{bounded}, that is,
$$\D=\bigcup_{i,j\in \mathbb{Z}}\hua{P}^{\perp}[i]\cap\hua{P}[j].$$
Then a bounded t-structure is determined by its heart.
Then the inclusion relation $\hua{P}_1\supset\hua{P}_2$ equips the set of hearts with a partial order by $\h_1\leq\h_2$, where $\h_1=\hua{P}_1\cap\hua{P}_1^{\perp}[1]$ and $\h_2=\hua{P}_2\cap\hua{P}_2^{\perp}[1]$.

For an abelian category $\hua{H}$, we call a pair $\<\hua{F},\hua{T}\>$ of full subcategories a torsion pair if $\hua{F}=\hua{T}^{\perp}$ and for every
object $N\in\h$ there is a short exact sequence $M\to N\to L$ in $\h$
with $M\in\hua{T}$ and $L\in \hua{F}$. We call $\hua{T}$ (resp. $\hua{F}$) the torsion part (resp. torsion free part) of the torsion pair and write $\h=\<\hua{F},\hua{T}\>$.

By \cite{HRS},
for any heart $\h$ in a triangulated category $\D$
with torsion pair $\<\hua{F},\hua{T}\>$,
there exists the following two hearts in $\D$ with torsion pairs
\[
    \h^\sharp=\<\hua{T},\hua{F}[1]\>,\quad \h^\flat=\<\hua{T}[-1],\hua{F}\>.
\]
We call $\h^\sharp$ the \emph{forward tilt} of $\h$
with respect to the torsion pair $\<\hua{F},\hua{T}\>$,
and $\h^\flat$ the \emph{backward tilt} of $\h$.
If $\hua{F}=\<S\>$ is generated by a simple object in an abelian category $\hua{H}$, then
$\hua{H}=\<\hua{F},{}^{\perp}\hua{F}\>$ is a torsion pair,
where $^{\perp}\hua{F}=\{L\in \hua{H}:\Hom_{\D}(L,M)=0 \text{ for all }M\in\hua{F}\}$. In this case, we say the forward tilt is \emph{simple} and denote it by $\tilt{\h}{\sharp}{S}$. Similarly we have simple backward tilt $\tilt{\h}{\flat}{S}$.

\subsection{Exchange graphs}\label{sec:exchange graphs}

The notion of simple titling leads to exchange graphs.
Define the \emph{total exchange graph} $\EG \D$ of a triangulated category $\D$
to be the oriented graph whose vertices are all hearts in $\D$
and whose edges correspond to simple forward tilts between them.
For any $\h_0\in\EG \D$, we consider the full subgraph of $\EG \D$
induced by the interval $\{\h\mid\h_0\leq\h\leq\h_0[1]\}$ and let
$\EG_3(\D, \h_0)$ be its principal component, that is, the connected component
consisting of hearts that can be iterated simple tilted from $\h_0$.
Denote by $\EG(\D,\h_0)$ the principal component of $\EG \D$, that is the connected component of $\EG \D$ which contains $\h_0$.

In this section we collect some results about the exchange graphs of the derived categories $\D(Q)$, $\D(\qq)$, $\D(\S)$ and $\D(\ss)$. Recall from \cite[Section 9.3]{K08} that $\D(\qq)$ admits a standard heart $\gq$, generated by the simple $\qq$-modules, which is equivalent to the the category $\hua{H}_Q$. We denote $\Sim\hq=\{S_i\}_{i\in Q_0}$ the set of simple modules in $\hq$ and by $\Sim\gq=\{T_i\}_{i\in Q_0}$ the set of simple objects in $\gq$.
Then \cite[Theorem~8.1]{KQ} says that the $L$-immersion $\hua{L}$ \eqref{map:L immersion-quiver} maps $\hua{H}_Q$ to $\hua{G}_Q$ with $\hua{L}(S_i)=T_i$.
Moreover there is an isomorphism (as graph) induced by $\hua{L}$:
\begin{equation}\label{eq:L_Q}
    {\imm_Q}_*\colon \EG_3(\D(Q), \hq)\cong\EG_3(\D(\qq), \gq).
\end{equation}

Similarly, $\D(\ss)$ admits a standard heart $\gs$, generated by the simple $\ss$-modules $\Sim\gs=\{T_{\i}\}_{\i\in \S_0}$, which is equivalent to the category $\hua{H}_\S$ of finite dimensional ($\F_q$) modules over $\F_q \S$ with simple modules $\Sim\ss=\{S_\i\}_{\i\in \S_0}$.
Then the $L$-immersion \eqref{map:L immersion-species} maps $\hua{H}_\S$ to $\hua{G}_\S$ with $\hua{L}(S_\i)=T_\i$ and induces an isomorphism (as graph)
\begin{equation}\label{eq:L_S}
    {\imm_{\S}}_*\colon \EG_3(\D(\S), \hs)\cong\EG_3(\D(\ss), \gs).
\end{equation}

\subsection{Spherical twist groups}

Now we recall the spherical twist group, which can be used to compare the exchange graphs.
Recall that, i.e. from \cite{ST}, an object $T$ in a $\k'$-linear triangulated category is a \emph{$3$-spherelike $\k'$-object} if
\[\Hom^{\bullet}(T, T)=\k' \oplus \k'[-3].\]
For example, any $T_i\in \Sim\gq$ is a $3$-spherelike $\k$-object in $\D(\qq)$ by \cite{K4} (cf. \cite{ST}).

A $3$-spherelike $\k'$-object $T$ in a Calabi-Yau-$3$ triangulated category $\D$ is called a $3$-spherical $\k'$-object, which induces the \emph{twist functor} \cite{ST}
\begin{gather*}
    \phi_T(X)=\Cone\left(X\to \Hom^\bullet(X,T)^\vee\otimes T\right)[-1]
\end{gather*}
with inverse
\begin{gather*}
    \phi_T^{-1}(X)=\Cone\left(\Hom^\bullet(T,X)\otimes T\to X\right)
\end{gather*}
in the auto-equivalence group $\Aut\D$ of $\D$.

Denote by $\Br\qq$ the \emph{spherical twist group} of $\D(\qq)$,
that is, the subgroup of $\Aut\D(\qq)$ generated by $\{\phi_{T}\}_{T\in\Sim\gq}$.
Let $\EG(\D(\qq),\gq)$ be the principal component of the exchange graph $\EG(\D(\qq))$ containing the standard heart $\hua{G}_Q$. Then by \cite[Theorem~8.6]{KQ}, the action of $\Br\qq$ on $\D(\qq)$ induces an action on $\EG(\D(\qq),\gq)$ which gives an isomorphism (as vertex set)
\begin{equation}\label{eq:br1}
\EG_3(\D(\qq), \gq)\cong\EG(\D(\qq),\gq)/\Br({\qq}).
\end{equation}

Similarly, for $\ss$ of a species $\S$, the spherical twist group $\Br\ss$ is defined and we have an isomorphism (as vertex set)
\begin{equation}\label{eq:br2}
\EG_3(\D(\ss), \gs)\cong\EG(\D(\ss),\gs)/\Br\ss.
\end{equation}

\subsection{A useful lemma}\label{app:a useful lemma}
In this section, we recall the following proposition which is proved in \cite[Proposition~5.4]{KQ} (see also in \cite[section 7.2]{KY}) when $\k'=\k$ be an algebraically-closed field. This can be straightforwardly generalized to the case $\k'=\F_q$, where the difference is that the tensor needs to be carefully dealt with.  Then we prove a lemma. They will be used in Section \ref{Sec:$F$-stable}.

\begin{proposition}
\label{pp:fini}
In a $\k'$-linear triangulated category $\D$,
let $S$ be a rigid simple object in a finite heart $\h$, that is a heart generated by finite set $\Sim\h$ by means of extensions. Then after a forward or backward simple tilt the new simples are
\begin{eqnarray}
  \Sim\tilt{\h}{\sharp}{S} &=&
    \{\tilt{\psi}{\sharp}{S}(X)\mid X\in \Sim\h,X\neq S\}\cup\{S[1]\}\;\;,
\label{eq:psp-1} \\
  \Sim\tilt{\h}{\flat}{S} &=&
    \{\tilt{\psi}{\flat}{S}(X)\mid X\in \Sim\h,X\neq S\}\cup\{S[-1]\}\;,
\label{eq:psp-2}
\end{eqnarray}
where
\begin{eqnarray}
\label{eq:psi+}
  \tilt{\psi}{\sharp}{S}(X)
&=&\Cone\left(X\to \Ext^1(X, S)^* \otimes_E S[1] \right) [-1],
\\ \label{eq:psi-}
  \tilt{\psi}{\flat}{S}(X)
&=&\Cone \left( \Ext^1(S, X)\otimes_E S[-1] \to X \right),
\end{eqnarray}
and $E=\End(S)$.
Thus $\tilt{\h}{\sharp}{S}$ and $\tilt{\h}{\flat}{S}$ are also finite.
\end{proposition}
\begin{remark}
Note that since $S$ is simple, $E$ is a division ring over $\k'$, so $E=\k$ when $\k'=\k$, and the above proposition is just Proposition~5.2 in \cite{KQ}.
\end{remark}
\begin{lemma}
\label{lem:setofsimple}
Let $\h$ be a heart in a triangulated category $\D$ with rigid simples $R_1,\ldots,R_m$
such that $\Hom^\bullet(R_i,R_j)=0$ for any $1\leq i\neq j \leq m$.
Then there is a torsion pair $\<\hua{F},\hua{T}\>$ with respect to which the forward tilt of
$\h$ equals to $\h_m$, where
\[
    \h_0=\h, \quad \h_i=\tilt{(\h_{i-1})}{\sharp}{R_{\rho(i)}}
\]
and $\rho$ is any (fixed) permutation of $m$ elements.
Denote the tilt $\h_m$ of $\h$ by $\tilt{\h}{\sharp}{R_1,\ldots,R_m}$.
\end{lemma}
\begin{proof}
Fix a permutation $\rho$.
By repeatedly using Proposition~\ref{pp:fini},
we inductively deduce that any simple in $\h_i$, for $0\leq i\leq m$,
admits a filtration of triangles with factors in $\Sim\h\cup\Sim\h[1]$.
Thus, its homology with respect to $\h$ lives only in degree zero and one.
By \cite[Lemma~5.6]{KQ}, this means $\h\leq \h_i\leq\h[1]$.
By \cite[Lemma~2.8]{Q}, we know that $\h_i$ is the forward tilt of $\h$
with respect to some torsion pair.
Then, using formulae in Proposition~\ref{pp:fini}, a direct calculation shows that
the tilt $\h_m$ is independent of the choice of permutation, as required.
\end{proof}

\subsection{Stability conditions}\label{app:stability conditions}
In this section we recall some basic notions and facts on stability conditions we will use in this paper. For precise definitions, we refer the reader to \cite{B1}.

A \emph{stability function} on an abelian category $\hua{C}$
is a group homomorphism $Z:\K(\hua{C})\to\kong{C}$ such that
for any object $0\neq M\in\hua{C}$, we have
$Z(M) =m(M) \cdot e^{\mathbf{i}\pi \phi(M)}$ for some
$m(M) \in \kong{R}_{>0}$ and $\phi(M) \in (0,1]$,
i.e. $Z(M)$ lies in the upper half-plane
\begin{gather}\label{eq:H}
    H=\{r\exp(\mathrm{i} \pi \theta)\mid r\in\kong{R}_{>0},0< \theta\leq 1 \}
    \subset\kong{C}.
\end{gather}

\begin{definition}
A \emph{stability condition} $\sigma$ on a triangulated category $\D$ consists of a heart $\h$ and a stability function $Z$ on $\h$ with the \emph{Harder-Narashimhan property},
which will be denoted by $\sigma=(\h,Z)$.
\end{definition}

The data in a stability condition $\sigma=(\hua{H},Z)$ is equivalent to the pair $(Z,\hua{P})$, where $\hua{P}$ is a \emph{slicing} $\hua{P}=\{\hua{P}(\phi)\mid\phi\in\mathbb{R}\}$,
whose canonical heart is $\hua{H}$ and is compatible with $Z$.
More precisely, the slicing is a collection of additive subcategories of a triangulated category $\D$ satisfying
\begin{itemize}
\item $\hua{P}(\phi+1)=\hua{P}(\phi)[1], \forall \phi\in\mathbb{R}$.
\item $\Hom(\hua{P}(\phi_1),\hua{P}(\phi_2))=0$ for $\phi_1>\phi_2$,
\item Any object $M$ in $\D$ admits a Harder-Narashimhan filtration, i.e.
\begin{equation}\label{eq:HN}
    0 =
    \xymatrix @C=5mm{
     M_0 \ar[rr]   &&  M_1 \ar[dl] \ar[rr] && M_2 \ar[dl]
     \ar[r] & \dots  \ar[r] & M_{n-1} \ar[rr] && M_n \ar[dl] \\
    & A_1 \ar@{-->}[ul] && A_2 \ar@{-->}[ul] &&&& A_n \ar@{-->}[ul]
    }
    = M
\end{equation}
with $A_i\in\hua{P}(\phi_i)$ satisfying $\phi_1 > \phi_2 > \cdots > \phi_n$. We denote by $\phi^+(M)=\phi_1$ and $\phi^-(M)=\phi_n$.
\end{itemize}
The compatibility condition between $\hua{P}$ and $Z$ mentioned above means
$$Z(M) =m(M) \cdot e^{\mathbf{i}\pi \phi(M)}$$ for any $M\in\hua{P}(\phi)$.
We will also write $\sigma=(Z, \hua{P})$.

Denoted by $\Stab\D$ the space of all stability conditions on
$\D$, then the function
\begin{equation}\label{equ:matric}
d(\sigma_1,\sigma_2)=\sup_{0\neq M\in\D}
\bigg\{|\phi^-_{\sigma_2}(M)-\phi^-_{\sigma_1}(M)|,|\phi^+_{\sigma_2}(M)-\phi^+_{\sigma_1}(M)|
,|\log \frac{m_{\sigma_2}(M)}{m_{\sigma_1}(M)}|\bigg\}
\end{equation}
defines a (generalised) metric on $\Stab(\D)$.
Recall a crucial result due to Bridgeland.

\begin{theorem}\label{thm:B1}\cite[Theorem~1.2]{B1}
Each connected component of $\Stab\D$ is locally homeomorphic to a linear sub-manifold of
$\Hom_{\kong{Z}}(\K(\D),\kong{C})$, sending a stability condition $(\h, Z)$ to its central change $Z$. In particular, $\Stab\D$ is a complex manifold.
\end{theorem}

Note that every finite heart $\h$
corresponds to a (complex, half closed and half open) $n$-cell
\[\cub(\h)\simeq H^n\]
inside $\Stab\D$, where $H$ is defined as in \eqref{eq:H}.

\subsection{Frobenius functors}
We mention that by \cite{DD1}, there is an algebra isomorphism $(\k Q)^F\cong\F_q\S$.
Further, by \cite{DD2} there is a Frobenius functor in $\Aut\D(Q)$ induced by $F$, which we also denote by $F$.
For any object $M$ in $\D(Q)$, denote by $p(M)$ the $F$-period of $M$,
that is, the minimal positive integer $m$, such that $F^m(M)=M$.
Then $p(M)$ is finite for any $M$ and let
\begin{gather}
\label{eq:tilde}
    \widetilde{M}=\bigoplus_{j=1}^{p(M)} F^j(M).
\end{gather}
We say $M$ is an $F$-stable object if $p(M)=1$, i.e. $F(M)=M$.
We say a subcategory $\hua{C}$ of $\D(Q)$ is $F$-stable if $F(\hua{C})=\hua{C}$.

Let $\D(Q)^F$ be the $F$-stable category of $\D(Q)$,
whose objects are $F$-stable objects of $\D(Q)$
and whose morphisms are the ones which commute with $F$
(see \cite[Remark~5.5]{DD2} for details). We point out that $\D(Q)^F$ is not a full subcategory of $\D(Q)$.
Then $\D(Q)^F$ is a triangulated category and there is a derived equivalence
\begin{gather}
\label{eq:Phi}
    \Phi\colon\D(\S)\cong\D(Q)^F
\end{gather}
such that
\[
    \Phi(\hs)=\hq^F\quad\text{and}\quad
    \Phi(S_\i)=\bigoplus_{i\in\i} S_i,
\]
where we view $\hq$ as a full subcategory of $\D(Q)$ and $\hq^F$ is the $F$-stable subcategory of $\hq$ which is generated by the $F$-stable objects in it.
By \eqref{eq:tilde}, the second equation is equivalent to $\Phi(S_\i)=\widetilde{S}_i$ for any $i\in\i$, and we refer to $\widetilde{S}_\i$.
Further, for any $X\in\Ind\D(Q)^F$, the set of indecomposables in $\D(Q)^F$,
there exists $M\in\Ind\D(Q)$, the set of indecomposables in $\D(Q)$, such that $X=\widetilde{M}$.

The Frobenius morphism $F$ (cf. Remark \ref{rem:Frobmorp of dg algebra}) on $\Gamma Q$ also induces a Frobenius functor, still denoted by $F$, on $\D(\qq)$, which is an auto-equivalence. Denote the $F$-stable category of $\D(\qq)$ as $\D(Q)^F$ by $\D(\qq)^F$.
Then we have the following diagram, where the dashed arrow will be defined in the next
section in such a way that the diagram commutes
\begin{gather}\label{eq:diagram}
\xymatrix@C=3pc{
    \D(\S) \ar[r]^{\Phi}_{\cong}\ar[d]^{\imm_\S} & \D(Q)^F \ar@{^{(}->}[r]\ar[d]^{\imm_Q} & \D(Q) \ar[d]^{\imm_Q}\\
    \D(\ss) \ar@{-->}[r]^? & \D(\qq)^F \ar@{^{(}->}[r] & \D(\qq).}
\end{gather}

\section{Folding Calabi-Yau category}\label{Sec:Folding}
In this section, we aim to complete \eqref{eq:diagram}.

\begin{proposition}\label{pp:Theta}
There is a faithful functor $\Theta\colon\D(\ss)\to\D(\qq)$,
sending $T_{\i}$ to $\widetilde{T_i}=\bigoplus_{i\in\i} T_i$ and inducing a derived equivalence
\begin{gather}
\label{eq:Theta}
    \Theta\colon\D(\ss)\cong\D(\qq)^F.
\end{gather}
\end{proposition}
\begin{proof}
By the equivalence $\gq\cong \hq$, $\widetilde{T_i}$ as in \eqref{eq:tilde} is well-defined for $i\in Q_0$ and
\[\Theta(T_\i):=\widetilde{T_i}\in\D(\qq)^F\]
for any $i\in\i$.

Moreover, we show that
$\{\Theta(T_\i)\}_{\i\in\S_0}$ generates $\D(\qq)^F$.
In fact, on the one hand, $\{\Theta(T_\i)\}_{\i\in\S_0}$ generates $\gq^F$ since $\gq\cong \hq$ and $\{\widetilde{S_\i}\}_{\i\in\S_0}$ generates $\hq^F$.
On the other hand, any $M\in\D(\qq)^F$ admits a canonical filtration
\begin{equation}
\label{eq:canonfilt}
\xymatrix@C=0,5pc{
  0=M_0 \ar[rr] && M_1 \ar[dl] \ar[rr] &&  \cdots\ar[rr] && M_{m-1}
        \ar[rr] && M_m=M \ar[dl] \\
  & H_1[k_1] \ar@{-->}[ul]  && && && H_m[k_m] \ar@{-->}[ul]
  }
\end{equation}
where $H_i \in \gq$ and $k_1 > \ldots > k_m$ are integers.
Since $M$ is $F$-stable and the filtration is unique, we deduce that
each \emph{homology} $H_i$ of $M$, with respect to $\gq$, is $F$-stable.
Hence, any of these homologies is generated by $\{\widetilde{T_\i}\}_{\i\in\S_0}$,
and so is $M$ as required.

Then we extend the definition $\Theta$ to any object in $\D(\ss)$ by using extensions and shifts of the heart $\gs$, which is generated by the simples $T_{\i}$, and then by using the HN-filtration of the object. Due to above discussion, $\Theta$ is a dense map from the set of objects in $\D(\ss)$ to the set of objects in $\D(\qq)^F$.

At last, we define $\Theta$ on the morphisms in $\D(\ss)$ through L-immersions \eqref{map:L immersion-species} and the equivalence $\Phi\colon\D(\S)\cong\D(Q)^F$ \eqref{eq:Phi}. Then $\Theta$ is an equivalence by the property \eqref{eq:lagrangian} of L-immersions and the fact that $\Phi$ is an equivalence, noticing indeed that \eqref{eq:lagrangian} has a natural splitting (see \cite[Lemma 4.4(b)]{K8}).
\end{proof}

Thus diagram \eqref{eq:diagram} can be completed to the following diagram, where the commutativity of the left square directly follows the definition of $\Theta$
\begin{gather}
\label{eq:diagramed}
    \xymatrix@C=3pc{
    \D(\S) \ar[r]^{\Phi}_{\cong}\ar[d]^{\imm_\S} & \D(Q)^F \ar@{^{(}->}[r]\ar[d]^{\imm_Q}
        & \D(Q) \ar[d]^{\imm_Q}\\
    \D(\ss) \ar[r]^{\Theta}_{\cong} & \D(\qq)^F \ar@{^{(}->}[r] & \D(\qq).}
\end{gather}

\section{$F$-stability for bounded derived categories}\label{Sec:$F$-stable}
In this section, we introduce the $F$-stable stability condition, and prove that when $(Q,\S)$ is of Dynkin type, we have a isomorphism $\FS\D(Q)\cong\Stab\D(\S)$, where $\FS\D(Q)$ is the space of $F$-stable stability condition on $\D(Q)$.

Note that the action of auto-equivalence $F$ on $\D(Q)$ induces an automorphism on $\K(\D(Q))$. Denote by $\K^F(\D(Q))$ the subgroup in $\K(\D(Q))$ consisting of $F$-stable elements.
Then we have a canonical homomorphism from $\K(\D(Q)^F)$ to $\K^F(\D(Q))$, which is an isomorphism by \cite[Theorem 7.2]{DD2}, see also \cite[Theorem 8.5]{DD2}. Further, $F$ also induces an action on $\Stab \D(Q)$ and we introduce the following

\begin{definition}\label{def:fstab stability conditions}
We call a stability condition $(\h, Z)$ in $\Stab\D(Q)$ $F$-stable if
\begin{itemize}
\item $\h$ is $F$-stable, that is $F(\h)=\h$;
\item $Z$ is $F$-stable, that is $Z(F(M))=Z(M)$.
\end{itemize}
Denote by $\FS\D(Q)$ the subspace in $\Stab\D(Q)$ consisting of $F$-stable stability conditions.
\end{definition}

Note that the equality $F(\h)=\h$ in the above definition is equivalent to an inclusion $F(\h)\subseteq \h$, since both $F(\h)$ and $\h$ are hearts of bounded t-structures.

For an $F$-stable heart $\h$, define its F-stabilization to be the full subcategory $\h^F$ in $\D(Q)^F$,
consisting of objects $\{\widetilde{M}\mid M\in\h\}$, where $\widetilde{?}$ is as in \eqref{eq:tilde}.
Then $\h^F$ is an abelian category, with short exact sequences in $\h^F$ being precisely the sum of short exact sequences in a same orbit.
\begin{lemma}
\label{lem:F-heart}
If $\h$ is an $F$-stable heart in $\D(Q)$, then its F-stabilization $\h^F$ is a heart in $\D(Q)^F$.
\end{lemma}
\begin{proof}
Assume $a>b$ are any integers.
We use the criterion in \cite[Lemma~3.2]{B1} for the definition of hearts,
namely, an abelian category $\hua{A}$ in a triangulated category $\D$ is a heart
if and only if it satisfies the following conditions
\begin{itemize}
\item $\Hom_{\D}(A[a],B[b])=0$ for any $A,B\in\hua{A}$.
\item There is a canonical filtration \eqref{eq:canonfilt} for any $M$ in $\D$ with
$H_i \in \hua{A}$ and $k_1 > \ldots > k_m$ are integers.
\end{itemize}

For any $X$ and $Y$ in $\h^F$,
there exists $M$ and $L$ in $\h$ such that $X=\widetilde{M}$ and $Y=\widetilde{L}$ as in \eqref{eq:tilde}.
Since $\Hom_{\D(Q)}(A[a],B[b])=0$ for any $A,B\in\h$,
we have $\Hom_{\D(Q)^F}(X[a],Y[b])=0$.

Further, there is a canonical filtration \eqref{eq:canonfilt} of $M$
with $H_i \in \h$ and $k_1 > \ldots > k_m$ are integers.
Since $\h$ is $F$-stable, $F^j(H_i)$ is also in $\h$ which implies
the canonical filtration of $F^j(M)$ has factors $F^j(H_1)[k_1],\ldots,F^j(H_m)[k_m]$.
By direct summing the triangles in the filtrations of $F^j(M)$,
for $j=1,\ldots,p(M)$, we obtain a filtration of $\widetilde{M}$ in $\D(Q)$,
with factors
\[
    \bigoplus_{j=1}^{p(M)} F^j(H_1)[k_1],\ldots,\bigoplus_{j=1}^{p(M)} F^j(H_m)[k_m].
\]
To see that this induces the canonical filtration of $X=\widetilde{M}$ in $\D(Q)^F$
(under $\Phi$), we only need to show that $\bigoplus_{j=1}^{p(M)} F^j(H_i)[k_i]$ is $F$-stable,
or equivalently, $F^{p(M)}(H_i)=H_i$.
This follows by comparing the canonical filtrations of $F^{p(M)}(M)$ and $M$,
noticing that the canonical filtration is unique.
Therefore, $\h^F$ is a heart in $\D(Q)^F$.
\end{proof}

We think that the converse of Lemma~\ref{lem:F-heart} is also true,
but we only (need and) prove a partial result, and we will prove the inverse in Corollary \ref{cor:inverse of the lemma} for the case of Dynkin type. An immediate corollary for stability conditions is as follows.

\begin{corollary}
\label{cor:stab}
There is a canonical inclusion
\begin{gather}
\label{eq:iota1}
    \tau_Q: \FS\D(Q) \to\Stab\D(Q)^F
\end{gather}
sending a stability condition from $(\h, Z)$ to $(\h^F, Z^F)$,
where $Z^F(\widetilde{M})=Z(M)$, for any $M\in\D(Q)$.
Moreover, we have the following commutative diagram
\begin{gather}\label{eq:commdia}
\xymatrix{
    \FS\D(Q)  \ar@{->}[d]\ar[r]^{\tau_Q}  & \Stab\D(Q)^F  \ar@{->}[d] \\
    \Hom(\K^F(\D(Q)), \kong{C}) \ar@{-}[r]^{\cong} &  \Hom(\K(\D(Q)^F), \kong{C}),
}\end{gather}
where the vertical maps send a stability condition to its central charge.
\end{corollary}
\begin{proof}
By Definition \ref{def:fstab stability conditions}, an $F$-stable stability condition $(\h, Z)$ satisfies $Z(F(M))=Z(M)$, so $Z^F$ is well-defined.
Thus we have the inclusion $\tau_Q$ and the canonical commutative diagram \eqref{eq:commdia}.
\end{proof}

Recall that $\EG(\D(Q),\h_Q)$ and $\EG(\D(\S),\h_\S)$ are the principal components of $\EG\D(Q)$ and $\EG\D(\S)$ respectively.
Note the derived equivalence \eqref{eq:Phi} induces
an isomorphism $\Phi\colon\EG\D(\S)\to\EG\D(Q)^F$ (we abuse the notation $\Phi$ here).

\begin{proposition}
\label{pp:F}
Any heart in $\Phi(\EG(\D(\S),\h_\S))$ is the F-stabilization of some heart in $\EG(\D(Q),\h_Q)$.
Moreover, if $\Phi(\h)=\hh^F$ for hearts $\h\in\EG(\D(\S),\h_\S)$ and $\hh\in\EG(\D(Q),\h_Q)$,
we have the following.
\numbers
\item $\Sim\h$ and $\Sim\hh$ can be written as $\{R_\i\}_{\i\in\S_0}$ and $\{R_i\}_{i\in Q_0}$,
respectively, such that $\Phi(R_\i)=\bigoplus_{i\in\i}R_i$.
\item For any $\i\in\S_0$ and $i_1,i_2\in\i$, $\Hom^\bullet(R_{i_1},R_{i_2})=0$.
\item For any $\i\in\S_0$, we have
\begin{gather}
\label{eq:tilts x}
    \Phi(\tilt{\h}{\sharp}{})=( \tilt{\hh}{\sharp}{} )^F\quad\text{and}\quad
    \Phi(\tilt{\h}{\flat}{})=( \tilt{\hh}{\flat}{} )^F,
\end{gather}
where the tilts of $\h$ are with respect to $R_\i$ and
the tilts of $\hh$ are with respect to the set of simples $\{R_i\}_{i\in\i}$
in the sense of Lemma~\ref{lem:setofsimple}.
\ends
\end{proposition}
\begin{proof}
We use induction starting from the standard hearts
$\Phi(\hs)=\hq^F$ satisfying $1^\circ$ and $2^\circ$.
We only need to show that, if $(\h,\hh)$ satisfy $\Phi(\h)=\hh^F$, $1^\circ$ and $2^\circ$,
then they also satisfy $3^\circ$ and hearts in \eqref{eq:tilts x} satisfy $1^\circ$ and $2^\circ$.

Let $(\h,\hh)$ satisfy $\Phi(\h)=\hh^F$, $1^\circ$ and $2^\circ$.
Consider a fixed $\i\in\S_0$.
Suppose that the orbit $\i$ contains vertices $1,\ldots,|\i|$ in $Q_0$.
We claim that
\begin{gather}
\label{eq:pp:F}
\tilt{(\widetilde{\h}^F)}{\sharp}{\widetilde{R}_{\i}}=
    \left(\tilt{\widetilde{\h}}{\sharp}{R_1,\ldots,R_{|\i|}}\right)^F,
\end{gather}
where $\widetilde{R_{\i}}=\widetilde{R_i}=\bigoplus_{i\in\i}R_i$.
Because of $2^\circ$,
RHS of \eqref{eq:pp:F} is well-defined in the sense of Lemma~\ref{lem:setofsimple}.
Since the simples determine a heart,
\eqref{eq:pp:F} is equivalent to the equality between their sets of simples.
Let $\j\in\S_0$ and contains vertices $1',\ldots,|\j|'$ in $Q_0$ with corresponding simples
$R_{j'}\in\Sim\h$. Let $\widetilde{R_{\j}}=\widetilde{R_{j'}}=\bigoplus_{j'\in\j}R_{j'}$.
By formulae in Proposition~\ref{pp:fini}, we only need to show
\begin{gather}\label{eq:RX}
    \tilt{\psi}{\sharp}{\widetilde{R_{\i}}}(\widetilde{R_{\j}})
    =\bigoplus_{j=1}^{|\j|} \Psi(R_{j'})
\end{gather}
where $\Psi=\tilt{\psi}{\sharp}{R_1}{\circ\cdots\circ\tilt{\psi}{\sharp}{R_{|\i|}}}$
and $\tilt{\psi}{\sharp}{}$ is defined as in \eqref{eq:psi+}.

Let $d=\gcd(|\i|, |\j|)$ and $|\i|=sd$, $|\j|=td$ for some integers $s,t$.
Without lose of generality, suppose that $F(R_{k})=R_{k+1}$ and $F(R_{k'})=R_{(k+1)'}$,
where $R_{k+sd}=R_{k}$ and $R_{(k+td)'}=R_{k'}$.
Further, suppose that
\[\Ext^1_{\D(Q)}(R_{k'}, R_{1})=\k^{h_{k}}\]
for $k=1,\ldots, d$.
Then, by applying the Frobenius functor, we have
\[\Ext^1_{\D(Q)}(R_{j'}, R_{i})=\k^{h_{(j-i+1)}},\]
where $h_{(x+d)}=h_x$ for any $x\in\kong{Z}$.
Using formula \eqref{eq:psi+},
a direct calculation shows that $\Psi(R_{j'})$ admit a filtration of triangles in $\D(Q)$,
with factors
\[
    R_1^{h_{j}} ,\ldots, R_{sd}^{h_{(j-sd+1)}} ,R_{j'},
\]
for any $1\leq j\leq td$, where $R^h$ means a direct sum of $h$ copies of $R$.
Noticing that $\Hom^\bullet(R_{i_1},R_{i_2})=0$ for different $i_1, i_2$, we actually have triangles
\[
    R_{j'}[-1] \to  \bigoplus_{i=1}^{sd} R_{i}^{h_{(j-i+1)}} \to \Psi(R_{j'}) \to R_{j'}.
\]
Direct summing these triangles gives a triangle
\begin{gather}\label{eq:psi x1}
    \widetilde{R}_{\j}[-1] \xrightarrow{\alpha}  \widetilde{R}_{\i}^{t\cdot h}
        \to \bigoplus_{j=1}^{|\j|} \Psi(R_{j'}) \to \widetilde{R}_{\j},
\end{gather}
in $\D(Q)$, where $h=\sum_{k=1}^d h_k$.
By the definition of $\tilt{\psi}{\sharp}{R_i}$,
$\alpha$ in \eqref{eq:psi x1} contains all maps in
$\Hom_{\D(Q)}(\widetilde{R}_{\j}[-1],\widetilde{R}_{\i}^{t\cdot h})$.

On the other hand, recall that morphisms in $D(Q)^F$ are those in $\D(Q)$ which commute with $F$, and we have
\begin{gather*}
    \Ext_{D(Q)^F}^1(\widetilde{R}_{\j}, \widetilde{R}_{\i})=
        \F_{q^{ t\cdot s\cdot d\cdot h}},\quad
    \End_{D(Q)^F}(\widetilde{R}_{\i}, \widetilde{R}_{\i})=\F_{q^{s\cdot d}}.
\end{gather*}
Then \eqref{eq:psi+} gives a triangle
\begin{gather}\label{eq:psi x2}
    \widetilde{R}_{\j}[-1] \xrightarrow{\alpha'}  \widetilde{R}_{\i}^{t\cdot h} \to
    \tilt{\psi}{\sharp}{\widetilde{R}_{\i}}(\widetilde{R}_{\j}) \to \widetilde{R}_{\j}
\end{gather}
in $\D(Q)^F$, where $\alpha'$ is the universal map.
Therefore \eqref{eq:psi x2} in $\D(Q)^F$ is induced from \eqref{eq:psi x1} in $\D(Q)$,
which implies \eqref{eq:RX}. Similarly for the case of backward tilting.

Via $\Phi$ in \eqref{eq:Phi}, we see that $(\hh,\h)$ satisfy $3^\circ$
and the hearts in \eqref{eq:tilts x} satisfy $1^\circ$ and $2^\circ$ as required.
\end{proof}

We have the following main result of this subsection.
\begin{theorem}\label{thm:stable1}
If $(Q,\S)$ is one of the Dynkin type in Example~\ref{ex:Dynkin},
then $\tau_Q$ in \eqref{eq:iota1} is an isomorphism.
In particular, we have a biholomorphism $\FS\D(Q)\cong\Stab\D(\S)$ , with $\FS\D(Q)$ as a sub-manifold of $\Stab\D(Q)$.
\end{theorem}
\begin{proof}
It is proved in \cite[Appendix~A]{Q} (also cf. \cite{KV}) that $\EG\D(Q)$ is connected if $Q$ is of Dynkin type. By using the same argument, $\EG\D(\S)$ is also connected if $\S$ is of Dynkin type, that is we have $\EG\D(\S)=\EG(\D(\S),\h_\S)$.

For the first claim, we only need to show that $\tau_Q$ is surjective, or equivalently,
that any heart in $\D(Q)^F$ is the F-stabilization of some heart $\h$ in $\D(Q)$.
This follows from Proposition~\ref{pp:F}, noticing that we have the isomorphism $\Phi\colon\EG\D(\S)\to\EG\D(Q)^F$ and the connectedness of graphs $\EG\D(Q)$ and $\EG\D(\S)$.

Then by the derived equivalence \eqref{eq:Phi} and the first claim, we have an isomorphism $\tau_Q: \FS\D(Q)\cong\Stab\D(\S)$. Further, this is a biholomorphism since the charge maps are local homeomorphisms which are compatible with $\tau_Q$, see diagram \eqref{eq:commdia}. So $\FS\D(Q)$ is a sub-manifold of $\Stab\D(Q)$.
\end{proof}

We have the following immediate corollary, which is a partial converse of Lemma~\ref{lem:F-heart}.
\begin{corollary}\label{cor:inverse of the lemma}
If $(Q,\S)$ is of Dynkin type,
then any heart in $\D(Q)^F$ is the $F$-stabilization of some $F$-stable heart of $\D(Q)$.
\end{corollary}

\section{$F$-stability for finite dimensional category}\label{Sec:Folding2}
We proceed to discuss $F$-stable hearts and stability conditions in $\D(\qq)$.
Notice that formulae \eqref{eq:psi+} and \eqref{eq:psi-}
coincide with twist functor formulae (cf. \cite[Remark~7.1]{KQ}).
Hence, similar to the proof of \eqref{eq:RX}, we have the following lemma.

\begin{lemma}
\label{lem:Br}
Let the orbit $\i\in\S_0$ consists of vertices $1,\ldots,|\i|$ in $Q_0$.
Then the auto-equivalence
\[\phi_{\i}=\phi_{T_1}\circ\cdots\circ\phi_{T_{|\i|}}\]
preserves $F$-stable objects in $\D(\qq)$
and hence induces an auto-equivalence $\phi_{\i}$ on $\D(\qq)^F$.
Moreover, under the derived equivalence \eqref{eq:Theta},
$\phi_{T_{\i}}$ corresponds to $\phi_{\i}$.
\end{lemma}

Denote by $\Br\qq^F$ the subgroup of $\Br\qq$ generating by $\{\phi_{\i}\}_{\i\in\S_0}$.
Hence, by abuse of notation $\Theta$ in diagram \eqref{eq:diagramed}, we have
\[\Theta(\Br\ss)=\Br\qq^F.\]
Recall that the derived equivalence \eqref{eq:Theta} induces an
isomorphism $\Theta\colon\EG\D(\ss)\to\EG\D(\qq)^F$.
Now we prove a similar result to Proposition~\ref{pp:F} for $\D(\qq)$.

\begin{proposition}
\label{pp:F CY}
Any heart in $\Theta(\EG(\D(\ss),\gs))$ is an F-stabilization of some heart in $\EG(\D(\qq),\gq)$.
\end{proposition}
\begin{proof}
First, by \eqref{eq:L_S}, for any heart $\widehat\h$ in $\EG_3(\D(\ss), \gs)$,
there exists $\h$ in $\EG_3(\D(\S), \hs)$ such that $\widehat\h={\imm_{\S}}_*(\h)$.
Then by Proposition~\ref{pp:F}, $\Phi(\h)=\hh^F$, for some $\hh\in\EG\D(Q)$.
Moreover, by looking at the homology of $\hh$ with respect to $\hq$ (cf. \cite[Lemma~5.6]{KQ}),
$\hh$ is actually in $\EG_3(\D(Q),\hq)$.
Further, the quotient map $\qq\to\k Q$, which induces the immersion $\imm_Q$,
commutes with the Frobenius morphisms on $\k Q$ and $\qq$.
Therefore we have
${\imm_Q}_*(\hh^F)=\left({\imm_Q}_*(\hh)\right)^F$.
Together, we have
\[
    \Theta(\widehat\h)={\imm_Q}_*(\Phi(\h))={\imm_Q}_*(\hh^F)=({\imm_Q}_*(\hh))^F,
\]
i.e. $\Theta(\widehat\h)$ is the F-stabilization of some heart in $\EG_3(\Gamma Q,\hq)$.

By Lemma~\ref{lem:Br}, we know that $\Br\qq^F$ preserves F-stability.
By \eqref{eq:br2} we know that $\EG_3((\D(\qq)^F),\h_Q^F)$ is a fundamental domain for
$\EG((\D(\qq)^F),\h_Q^F)/\Br\qq^F$.
Thus each heart in $\EG(\D(\ss),\gs)$ is the F-stabilization of
some heart in $\EG(\D(\qq),\gq)$, as required.
\end{proof}

Similar to \cite[Corollary~5.3]{Q}, there are principal components
\begin{gather*}
    \Stap\D(\qq)=\bigcup{}_{\h\in\EG(\D(\qq),\gq)}\cub(\h),\\
    \Stap\D(\ss)=\bigcup{}_{\h\in\EG(\D(\ss),\gs)}\cub(\h)
\end{gather*}
in $\Stab\D(\qq)$ and $\Stab\D(\ss)$, respectively.
Denote by $\FStap\D(\qq)$ the subspace in $\Stap\D(\qq)$ consisting of $F$-stable stability conditions.
As in Theorem~\ref{cor:stab},
we have an immediate consequence of Proposition~\ref{pp:F CY}.

\begin{theorem}
\label{cor:stab CY}
If $(Q,\S)$ is one of the Dynkin types in Example~\ref{ex:Dynkin},
then there is a canonical isomorphism
\begin{gather}
\label{eq:iota2}
    \iota_{\qq}\colon \FStap\D(\qq) \cong \Stap\D(\qq)^F.
\end{gather}
Thus we have a biholomorphism $\FStap\D(\qq)\cong\Stap\D(\Gamma\S)$.
\end{theorem}
\begin{proof}
The map $\iota_{\qq}$ is constructed via the canonical isomorphism between
the Grothendieck groups of $\K^F(\D(\qq))$ and $\K(\D(\qq)^F)$
, cf. the following commutative diagram
\begin{gather}\label{eq:commdia2}
\xymatrix{
    \FStap\D(\qq)  \ar@{->}[d]\ar[r]^{\iota_{\qq}}  & \Stap\D(\qq)^F  \ar@{->}[d] \\
    \Hom(\K^F(\D(\qq)), \kong{C}) \ar@{-}[r]^{\cong} &  \Hom(\K(\D(\qq)^F), \kong{C}).
}\end{gather}
To prove $\iota_{\qq}$ is an isomorphism, we only need to show that it is surjective, or equivalently,
that any heart in $\EG(\D(\qq)^F,\gq)$ is the F-stabilization of some heart in $\EG(\D(\qq),\gq)$.
This follows from Proposition \ref{pp:F CY}.

The isomorphism $\FStap\D(\qq) \cong \Stap\D(\Gamma\S)$
follows immediately from the derived equivalence \eqref{eq:Theta} and the first claim. Similar to the proof of Theorem \ref{thm:stable1}, it is a biholomorphism.
\end{proof}

\section{Numerical stability conditions}
In this section, we study the space of numerical stability condition of $\D(\qq)$ via
the stability conditions of $\D(\ss)$, for two special Dynkin types,
namely $(Q,\S)$ is of type $(A_3, B_2)$ or $(D_4, G_2)$. We will show that for the types $A_3$ and $D_4$, the numerical stability conditions exactly coincide with the $F$-stable stability conditions.

Recall that the \emph{Euler form} on the Grothendieck group $\K(\D)$ of a triangulated category $\D$ is defined by
\begin{gather}\label{eq:euler form}
    \chi(M, L)=\sum_{i} (-1)^i dim\Hom^i(M,L),
\end{gather}
for any $M,L\in \D$ (here we abuse notations for elements in the category and in the Grothendieck group).
A numerical stability condition on $\D$ is a stability condition $(\h,Z)$
such that the central change $Z: \K(\D)\to \kong{C}$ factors through
the numerical Grothendieck group $\K(\D)/\mathrm{Z}_{\chi}(\D)$, where
\[
    \mathrm{Z}_{\chi}(\D)=\{X\in \K(\D) \mid \chi(X, Y)=0, \forall Y \in \K(\D) \}.
\]
Denote by $\NS\D$ the space of numerical stability conditions that are in $\Stab\D$.

Let $\NS\D(\qq)$ be the space of numerical stability conditions of $\D(\qq)$.
Denote by $\NSp\D(\qq)$ its principal component, that is, the connected component containing
the numerical stability condition with heart $\gq$.
We have the following main theorem.
\begin{theorem}
\label{thm:main}
For $(Q,\S)$ is of type $(A_3, B_2)$ and $(D_4, G_2)$ as in Example~\ref{ex:Dynkin},
$\NS\D(\qq)$ consists of $\Br\qq/\Br\ss$ many (connected) components,
each of which is isomorphic to
\[
    \NSp\D(\qq)=\FStap\D(\qq)\cong\Stap\D(\ss).
\]
\end{theorem}
\begin{proof}
We only deal with the case $(Q,\S)$ of type $(A_3, B_2)$, while the other case is similar.
Without loss of generality, suppose that the labeling and the orientations of $(Q, \S)$ are
\begin{gather}\label{eq:B_2}
    Q\colon \quad
    \xymatrix@R=0.1pc@C=1.7pc{
        & 1\quad \ar@/^/@{<->}[dd]^{\iota}\\
        2 \ar[dr]\ar[ur]\\
        & 3\quad \\
    }
    \qquad\qquad
    \S\colon \quad
    \xymatrix@R=0.1pc@C=1.7pc{
        \\
        \mathbf{2} \ar[r]& \mathbf{1}\\
        _1 & _2
    }
\end{gather}
Recall that $\Sim\gq=\{T_1,T_2,T_3\}$.
By a direct calculation, we know that
\begin{itemize}
\item $F(T_1)=T_3$, $F(T_2)=T_2$, $F(T_3)=T_1$ and hence $F^2=\id$.
Thus a stability condition $(\h,Z)$ is $F$-stable if and only if $\h$ is $F$-stable and
$Z(T_1)=Z(T_3)$.
\item $\mathrm{Z}_\chi(\qq)$ is generated by $[T_1]-[T_3]$.
Thus a stability condition $(\h,Z)$ is numerical if and only if
$Z(T_1)=Z(T_3)$.
\end{itemize}
Clearly, an $F$-stable stability condition is numerical.

Next, we investigate stability conditions in
\begin{gather}
\label{eq:FD}
    \hua{S}:=\bigcup{}_{\h\in\EG_3(\D(\qq),\gq)}\cub(\h).
\end{gather}
The Auslander-Reiten quiver of $\gq$ is as following
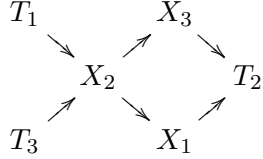
\[
\xymatrix@R=.7pc@C=.7pc{
    T_1 \ar[dr] && X_3 \ar[dr] \\
    & X_2 \ar[ur]\ar[dr] && T_2  \\
    T_3 \ar[ur] && X_1 \ar[ur] }
\]
$\EG_3(\D(\qq),\gq)$ is shown in Figure~\ref{fig:1},
where we denote each heart by a complete set of simples.
Note that the $F$-stable ones are underlined.

\begin{figure}\[
{
\xymatrix@R=2pc@C=0.7pc{
    &&_{\{T_1,X_3[1],T_2\}} \ar[drr]  \ar[d] \\
    _{\{T_1,T_3[1],X_3\}} \ar[urr] \ar@{<-}[dd] \ar[dr]
    && _{\{ T_1[1] X_1 X_3[1] \}}   \ar[dr]
    && _{\{T_3[1],T_1,T_2[1]\}}   \ar@{<-}[dr]   \ar[dd]  \\
    & \underline{_{\{T_1[1],X_2,T_3[1]\}}}    \ar[dr]   \ar@{<-}[dd]
    && \underline{_{\{T_2,X_1[1],X_3[1]\}}}   \ar[dr]
    && \underline{_{\{T_1,T_2[1],T_3\}}}   \ar[dd] \\
    \underline{_{\{T_1,T_2,T_3\}} }     \ar[dr] \ar@{-->}[rrrrru]
    && \underline{_{\{X_1,X_2[1],X_3\}} }  \ar[uu]   \ar[dr]
    && \underline{_{\{T_1[1],T_2[1],T_3[1]\}}}   \ar@{<-}[dr] \\
    & _{\{T_3,T_1[1],X_1\}}    \ar[drr]
    && _{\{ X_1[1], X_3, T_3[1] \}}   \ar@{<-}[d] \ar[uu]
    && _{\{T_2[1],T_3,T_1[1]\}}   \\
    &&& _{\{T_2,X_1[1],T_3\}}  \ar@{->}[urr] &
    }}
\]
\caption{The exchange graph $\EG_3(\qq, \gq)$ for $Q$ of type $A_3$}
\label{fig:1}
\end{figure}

Let $\h$ be a non-$F$-stable heart in Figure~\ref{fig:1}.
We claim that all the stability conditions in $\cub(\h)$ are not numerical.
To see this, take the top heart in Figure~\ref{fig:1} for example.
Then $Z(T_1)$ and $Z(T_3[1])$ are in the same upper half plane $H$ as in \eqref{eq:H}.
Thus $Z(T_1)=Z(T_3)$ never holds, which implies the claim. The other cases
are similar.

To sum up, for an $F$-stable heart $\h$ in Figure~\ref{fig:1}, a stability condition with heart $\h$ is numerical if and only if it is $F$-stable, i.e.
\begin{gather}
\label{eq:N=F}
    \cub(\h)\cap\NS\D(\qq)=\cub(\h)\cap\FStap\D(\qq).
\end{gather}
Therefore, we have
\[
    \hua{S}_*:=\hua{S} \cap \NS\D(\qq)=\hua{S} \cap\FStap\D(\qq).
\]
Notice that $\Br\qq^F$ preserves $F$-stable stability conditions
and all auto-equivalence $\Aut\D(\qq)$ preserves numerical stability conditions.
Then by \eqref{eq:br1} we have
\begin{eqnarray*}
    \NS\D(\qq)=\Br\qq\cdot\hua{S}_*.
\end{eqnarray*}
Similarly, by \eqref{eq:br2} and \eqref{eq:Theta} we have
\[     \EG_3(\D(\qq)^F, \gq^F)\cong\EG\D(\qq)^F/\Br{\qq} ^F \]
and hence
\begin{eqnarray*}
    \FStap\D(\qq)=\Br\qq^F\cdot\hua{S}_*.
\end{eqnarray*}
Thus $\NS\D(\qq)$ is the union of $\Br\qq/\Br\qq^F\cong \Br\qq/\Br\ss$ many copies of $\FStap\D(\qq)$.

To finish, we assert that the closure of $\FStap\D(\qq)$, which is taken inside $\Stap\D(\qq)$,
is disjoint from
\[
    C_0:=\NS\D(\qq)-\FStap\D(\qq).
\]
and vice versa.
This will imply that $\FStap\D(\qq)$ is a connected component of $\NS(\qq)$
and hence the theorem follows.

The rest of the proof is devoted to prove the assertion.
Let $\EG^F=\EG(\D(\qq)^F,\gq^F)$
and $\cub^F(\h)=\cub(\h)\cap\FStap\D(\qq)$ for any $\h\in\EG^F$.
First, we have
\[
    \cc{\FStap\D(\qq)}=\bigcup{}_{\h\in\EG^F}  \cc{\cub^F(\h)}
\]
by the local closeness property (\cite[Thm.~A and Lem.~3.26]{QW})
and thus we only need to show that, for any $\h\in\EG^F$,
\begin{gather}
\label{eq:CC}
    \cc{\cub^F(\h)} \cap C_0=\emptyset.
\end{gather}
Without loss of generality, take the $F$-stable heart $\h=\gq[1]$.
By formula \cite[(3.1)]{Q}, we have
\begin{gather}
\label{eq:end1}
    \cc{\cub^F(\gq[1])}\subset\cc{\cub(\gq[1])}
    \subset \bigcup{}_{\h\in\EG_3(\D(\qq),\gq) } \cub(\h).
\end{gather}
We also have
\begin{gather}
\label{eq:end2}
    C_0\subset \bigcup{}_{\h\in\Br\qq\cdot\EG^F-\EG^F} \cub(\h).
\end{gather}
But \eqref{eq:br1} implies that
\[
    \Br\qq\cdot\EG^F\cap\EG(\D(\qq),\gq)=\EG^F\cap\EG(\D(\qq),\gq)\subset\EG^F
\]
and hence
\begin{gather}
\label{eq:end3}
    \left(\Br\qq\cdot\EG^F-\EG^F\right) \cap \EG(\D(\qq),\gq)
    =\emptyset.
\end{gather}
Combine \eqref{eq:end1}, \eqref{eq:end2} and \eqref{eq:end3},
we have \eqref{eq:CC} for $\h=\gq[1]$ as required.
Similarly we can show that ${\cub^F(\h)} \cap \cc{C_0}=\emptyset$,
noticing $C_0$ is the union of many copies of $\cub^F(\h)$
(but the argument is the same
).
\end{proof}

\begin{remark}
If $(Q,\S)$ is of type $(A_3, B_2)$ as in \eqref{eq:B_2},
then $\Br\ss$ satisfies the $B_2$-braid relation, i.e.
\[
    (\mathbf{1}\circ\mathbf{2})^2=\mathbf{1}\circ\mathbf{2}\circ\mathbf{1}\circ\mathbf{2}
    =\mathbf{2}\circ\mathbf{1}\circ\mathbf{2}\circ\mathbf{1}=(\mathbf{2}\circ\mathbf{1})^2,
\]
where $\mathbf{i}$ represent the twist functor of $T_{\i}$ in $\Br\ss$.
This follows by Lemma~\ref{lem:Br} and a direct calculation
\begin{eqnarray*}
     (1\circ3)\circ2\circ(\underline{1\circ3})\circ2
    &=&1\circ\underline{3\circ2\circ3}\circ1\circ2\\
    =1\circ2\circ3\circ\underline{2\circ1\circ2}
    &=&1\circ2\circ\underline{3\circ1}\circ2\circ1\\
    =\underline{1\circ2\circ1}\circ3\circ2\circ1
    &=&2\circ1\circ\underline{2\circ3\circ2}\circ1\\
    =2\circ1\circ3\circ2\circ\underline{3\circ1}
    &=&2\circ(1\circ3)\circ2\circ(1\circ3)
\end{eqnarray*}
where $i$ represent the twist functor of $T_i$ in $\Br\qq\cong\Br_{A_3}$, and we underline the relations.
Thus
\[
    \Br_{B_2}\cong\Br\ss \subset \Br\qq\cong\Br_{A_3}.
\]
Similarly, If $(Q,\S)$ is of type $(D_4, G_2)$,
then $\Br\ss$ satisfies the $G_2$-braid relation, i.e.
\[
    (\mathbf{1}\circ\mathbf{2})^3
    =(\mathbf{2}\circ\mathbf{1})^3
\]
and we have
\[
    \Br_{G_2}\cong\Br\ss \subset \Br\qq\cong\Br_{D_4}
\]
(note that we need the faithfulness \cite[Thm.~B]{QW} of $\Br\qq\cong\Br_{D_4}$
in this case).
\end{remark}

\begin{remark}
Note that for a general Dynkin quiver $Q$, the numerical Grothendieck group $\mathrm{Z}_\chi(\qq)$ does not have to match the Grothendieck group of $\D(\ss)$. Thus the results in this section do not hold in general.
\end{remark}

\section{Stability conditions of Gepner type}\label{subsection:Stability conditions of Gepner type}
In this section, we consider the Gepner type stability conditions on the bounded derived category of a hereditary algebra of Dynkin type.
Firstly we recall some background on the Gepner type stability conditions.

There is a natural $\mathbb{C}$ action
on a space $\Stab(\D)$, namely:
\[
   s \cdot (Z,\hua{P})=(Z \cdot e^{ -\mathbf{i}\pi s},\hua{P}_{\Re(s)}),
\]
where $\hua{P}_x(\phi)=\hua{P}(\phi+x)$.
There is also a natural action on $\Stab(\D)$ induced by $\Aut(\D)$, namely:
$$\Phi  (Z,\hua{P})=(Z \circ \Phi^{-1} ,\Phi\circ \hua{P}).$$
For a pair $(\Phi,s)\in\Aut\D\times\mathbb{C}$, a {\em Gepner equation} on $\Stab\D$ is an equation of the form
\[\Phi(\sigma)=s\cdot\sigma.\]

In the usual setting, where one has $K(\D)\cong\mathbb{Z}^n$, Gepner
equations were studied by Toda \cite{T1,T2,T}, who was interested
in finding stability conditions with a symmetry property.
In this case, a solution to such an equation is called a Gepner point, which is an orbifold point in $\Aut\backslash\Stab\D/\mathbb{C}$.
Toda's motivation came from the study of Donaldson--Thomas invariants, in particular
for B-branes on Landau-Ginzburg models associated to a superpotential.

Following \cite{T}, a stability condition $\sigma$ is said to be of \emph{Gepner type $(\Phi, s)$} for $(\Phi, s)\in\Aut\D\times\mathbb{C}$, if they satisfy the Gepner equation.
In particular, the (n-)shift functor $[n]$ is an auto-equivalence and any stability condition is of Gepner type $([n],n)$, for any integer $n$. We call it the stability condition of trivial Gepner type.

We also mention the global dimension function on stability conditions,
which is closely related to Gepner equations, which is introduced in \cite{Q3}.
For a stability condition $\sigma=(Z, \hua{P})$, define
\begin{gather}\label{eq:geq}
\gldim\sigma=\sup\{ \phi_2-\phi_1 \mid
    \Hom(\hua{P}(\phi_1),\hua{P}(\phi_2))\neq0\}.
\end{gather}

In the following we assume $(Q,\S)$ is of Dynkin type. We list the Coxeter number $h$ of the Dynkin types in Table \ref{table:Coxeter number}, which will be used later.
\begin{table}[ht]
\begin{equation*}
\begin{array}{cc}
\textrm{Dynkin type} & \textrm{Coxeter number $h$}\\
\hline
A_{2n-1} & 2n \\
 B_n & 2n \\
 C_n & 2n \\
 D_{n+1} & 2n\\
E_6 & 12\\
F_4  & 12\\
G_2 & 6\\
\end{array}
\end{equation*}
\smallskip
\caption{The Coxeter numbers}
\label{table:Coxeter number}
\end{table}

Recall that one may associate to the module category $\h_Q$ a quiver, the Auslander-Reiten quiver $\Delta(\h_Q)$, whose vertices are labeled by the isomorphism classes of indecomposable modules and arrows are labeled by irreducible morphisms between the modules. For $\h_\S$, since $\F_q$ is not algebraically-closed, the Auslander-Reiten quiver $\Delta(\h_\S)$ is a species, which is folded from $\Delta(\h_Q)$ by an (admissible) automorphism $\overline{\iota}$ on $\Delta(\h_Q)$, where $\overline{\iota}$ is induced by the Frobenius morphism on $\h_Q$, see \cite[section 8]{DD1} for details.
Further, one may define the Auslander-Reiten quivers $\Delta(\D(Q))$ and $\Delta(\D(\S))$ of the bounded derived categories $\D(Q)$ and $\D(\S)$ respectively, which are spliced from $\Delta(\h_Q)$ and $\Delta(\h_\S)$ respectively.
Moreover, noticing that $\S$ is the folding of $Q$ by the admissible automorphism $\iota\in \Aut Q$ in Example \ref{ex:Dynkin}, we have a canonical exact sequence
\begin{gather}\label{eq:exact-seq1}
1 \longrightarrow \Aut Q \longrightarrow \Aut\Delta(\D(Q)) \longrightarrow \Aut\Delta(\D(\S)) \longrightarrow 1
\end{gather}
of automorphism groups of these quivers.

Note that a triangle-auto-equivalence of $\D(Q)$ (resp. $\D(\S)$) preserves the irreducible morphisms and
the Auslander-Reiten translation, thus it induces an automorphism of $\Delta(\D(Q))$ (resp. $\Delta(\D(\S))$).
So there is a homomorphism $\Aut\D(Q)\rightarrow \Aut\Delta(\D(Q))$ which is easily seen to be injective.
In particular, the Auslander-Reiten translation $\tau$ and the Frobenius functor $F$ are both triangle-auto-equivalences of $\D(Q)$, which induce two automorphisms on $\Delta(\D(Q))$, and in fact, any automorphism of $\Delta(\D(Q))$ is generated by them, excepting type $D_4$. For $D_4$ type, there are other auto-equivalences induced by the symmetries of the quiver, which also induce isomorphisms of $\Delta(\D(Q))$. So $\Aut\D(Q)\cong \Aut\Delta(\D(Q))$, and thus also $\Aut\D(\S)\cong \Aut\Delta(\D(\S))$. Further by
sequence \eqref{eq:exact-seq1}, the following canonical sequence is exact
\begin{gather}\label{eq:exact-seq2}
1 \longrightarrow \Aut Q \longrightarrow \Aut\D(Q) \longrightarrow \Aut\D(\S) \longrightarrow 1.
\end{gather}
Note that $\tau$ and $F$ commute, so for $Q$ is not of $D_4$ type, we have
\begin{equation}\label{eq:auto group}
    \Aut\D(Q)=\<\tau,F\>\cong \Aut\Delta(\D(Q))\cong
        \mathbb{Z} \times \mathbb{Z}_2.
\end{equation}
When $Q$ is of $D_4$ type, we have
\begin{equation}\label{eq:auto group3}
\Aut\D(Q)\cong \Aut\Delta(\D(Q))\cong
\mathbb{Z} \times S_3.
\end{equation}
By \eqref{eq:exact-seq2}, \eqref{eq:auto group} and \eqref{eq:auto group3}, we have
\begin{equation}\label{eq:auto group2}
\Aut\D(\S)=\<\tau\>\cong\mathbb{Z}.
\end{equation}

By the definition of the Gepner type stability condition,
it is not hard to see that a stability condition on $\D(Q)$ (of any quiver $Q$, not necessarily of Dynkin type) is $F$-stable if and only if it is of Gepner type $(F,0)$. Thus we have the following
\begin{proposition}\label{prop:Gepner3}
A stability condition in $\FS \D(Q)$ is of Gepner type $(F,0)$. Conversely, any Gepner type stability condition of $\D(Q)$ with auto-equivalence $F$ belongs to $\FS \D(Q)$.
\end{proposition}

It is proved in \cite{KST} that there is a stability condition $\sigma$ on $\D(Q)$ of Gepner type $(\tau, -2/h)$ (which is conjectured by \cite[Conjecture 1.2]{T1} for general settings) such that
\begin{gather}\label{eq:Gepner}
    \tau(\sigma)=-\frac{2}{h} \cdot \sigma.
\end{gather}
It is also the unique Gepner type stability condition on $\D(Q)$ up to $\mathbb{C}$-action with $\tau$ as the auto-equivalence. In \cite{KST}, $\sigma$ is defined on the homotopy category of graded matrix factorizations, which is equivalent to $\D(Q)$. Here we compute it directly in $\D(Q)$ and then fold it as a Gepner type stability condition in $\D(\S)$.

\begin{proposition}\label{prop:Gepner}
The following maps on the simples of $\h_Q$ induces a function $Z$ from $K(\D(Q))$ to $\mathbb{C}$, which gives a stability condition $(\h_Q,Z)$ on $\D(Q)$. This is the unique one of Gepner type $(\tau, -2/h)$ up to $\mathbb{C}$-action. The orientations and vertex labels of $Q$ is depicted in Figure \ref{fig:orientation}.
\begin{itemize}
\item
$A_{2n-1}$ type:
\[
    \begin{cases}
Z(T_{2m})=-1 & 1 \leq m \leq n-1;\\
Z(T_{2m-1})=\frac{e^{\mathrm{i}\pi/h}}{cos(\pi/h)}& 2 \leq m \leq n-1;\\
Z(T_1)=Z(T_{2n-1})=\frac{e^{\mathrm{i}\pi/h}}{2cos(\pi/h)}. &
\end{cases}
\]

\item
$D_{2n}$ type:
\[
    \begin{cases}
Z(T_{2m})=-1 & 1 \leq m \leq n-1;\\
Z(T_{2m-1})=\frac{e^{\mathrm{i}\pi/h}}{cos(\pi/h)}& 2 \leq m \leq n-1;\\
Z(T_1)=Z(T_{2n-1})=Z(T_{2n})=\frac{e^{\mathrm{i}\pi/h}}{2cos(\pi/h)}. &
\end{cases}
\]

\item
$D_{2n+1}$ type:
\[
    \begin{cases}
Z(T_{2n+1})=-1; &\\
Z(T_{2m})=-1 & 1 \leq m \leq n-1;\\
Z(T_{2m-1})=\frac{3e^{\mathrm{i}\pi/h}}{2cos(\pi/h)}& 2 \leq m \leq n;\\
Z(T_1)=\frac{e^{\mathrm{i}\pi/h}}{2cos(\pi/h)}. &
\end{cases}
\]

\item
$E_6$ type:
\[
    \begin{cases}
Z(T_{2})=Z(T_{4})=Z(T_{5})=-1; &\\
Z(T_1)=Z(T_6)=\frac{e^{\mathrm{i}\pi/h}}{2cos(\pi/h)};&\\
Z(T_{3})=\frac{3e^{\mathrm{i}\pi/h}}{2cos(\pi/h)}. &
\end{cases}
\]
\end{itemize}
\begin{figure}

\[
\quad
    \xymatrix@R=0.3pc@C=1.8pc{
~~A_{2n-1}& ~~1 \ar@{<-}[r]& 2 \ar@{->}[r]& 3 \ar@{<-}[r]& \cdots \ar@{<-}[r]& 2n-2\ar@{->}[r]& 2n-1}
\]

\[
\quad
    \xymatrix@R=0.3pc@C=1.8pc{
        &&&&& 2n-1 &\\
       ~~~~D_{2n}~~~~~& ~~~1 \ar@{<-}[r]& 2 \ar@{->}[r]& \cdots \ar@{<-}[r]& 2n-2 \ar@{->}[dr]\ar@{->}[ur]\\
        &&&&& 2n &\\
    }
\]

\[
\quad
    \xymatrix@R=0.3pc@C=1.8pc{
        &&&&& 2n &\\
       D_{2n+1}& 1 \ar@{<-}[r]& 2 \ar@{->}[r]& \cdots \ar@{<-}[r]& 2n-1 \ar@{->}[dr]\ar@{->}[ur]\\
        &&&&& 2n+1 &\\
    }
\]

\[
\quad
    \xymatrix@R=1.4pc@C=1.8pc{
              ~~E_{6}~~& ~~~1 \ar@{<-}[r]& 2 \ar@{->}[r]& 3\ar@{<-}[r]& 5\ar@{->}[r]& 6&&&\\
&&&4\ar@{->}[u]}
\]
\caption{The orientations of the quivers}
\label{fig:orientation}
\end{figure}
\end{proposition}
\begin{proof}
Note that $Z$ induces a well-defined function on $K(\h_Q)$. Since $\h_Q$ is a finite heart, $(\h_Q,Z)$ is a stability condition on $\D(Q)$.
By considering the Auslander-Reiten quiver of $\D(Q)$, one may directly check that $(\h_Q,Z)$ is of Gepner type $(\tau,-2/h)$, noticing that the value of $Z$ on a simple injective is $-1$. The uniqueness follows from \cite[Theorem 4.7]{Q}.
\end{proof}

\begin{proposition}\label{prop:Gepner2}
The stability condition $(\h_Q,Z)$ in Proposition \ref{prop:Gepner} is $F$-stable, and it induces a stability condition $(\h_\S,\underline{Z})$ on $\D(\S)$, which is of Gepner type $(\tau, -2/h)$, where $\tau$ is the Auslander-Reiten translation of $\D(\S)$ and $h$ is the Coxeter number of $\S$. It is also the unique Gepner type stability condition on $\D(\S)$ up to $\mathbb{C}$-action with $\tau$ as the auto-equivalence.
\begin{itemize}
\item
$C_{n}$ type:
\[
    \begin{cases}
\underline{Z}(S_{m})=-1 &  \text {if $m$ is even;}\\
\underline{Z}(S_{m})=\frac{e^{\mathrm{i}\pi/h}}{cos(\pi/h)}&  \text {if $m$ is odd.}
\end{cases}
\]

\item
$B_{2n-1}$ type:
\[
    \begin{cases}
\underline{Z}(S_{2m})=-1 & 1 \leq m \leq n-1;\\
\underline{Z}(S_{2m-1})=\frac{e^{\mathrm{i}\pi/h}}{cos(\pi/h)}& 2 \leq m \leq n-1;\\
\underline{Z}(S_1)=\underline{Z}(S_{2n-1})=\frac{e^{\mathrm{i}\pi/h}}{2cos(\pi/h)}. &
\end{cases}
\]

\item
$B_{2n}$ type:
\[
    \begin{cases}
\underline{Z}(S_{2m})=-1 & 1 \leq m \leq n;\\
\underline{Z}(S_{2m-1})=\frac{3e^{\mathrm{i}\pi/h}}{2cos(\pi/h)}& 2 \leq m \leq n;\\
\underline{Z}(S_1)=\frac{e^{\mathrm{i}\pi/h}}{2cos(\pi/h)}. &
\end{cases}
\]

\item
$F_4$ type:
\[
    \begin{cases}
\underline{Z}(S_{2})=\underline{Z}(S_{4})=-1 &\\
\underline{Z}(S_1)=\frac{e^{\mathrm{i}\pi/h}}{2cos(\pi/h)};&\\
\underline{Z}(S_{3})=\frac{3e^{\mathrm{i}\pi/h}}{2cos(\pi/h)}. &
\end{cases}
\]

\item
$G_2$ type:
\[
    \begin{cases}
\underline{Z}(S_{2})=-1; \\
\underline{Z}(S_1)=\frac{e^{\mathrm{i}\pi/h}}{2cos(\pi/h)}.
\end{cases}
\]
\end{itemize}
\end{proposition}
\begin{proof}
Note that $\h_Q$ and $Z$ are both $F$-stable, so $(\h_Q,Z)$ is $F$-stable. Then by Theorem \ref{thm:stable1}, it induces a stability condition $(\h_\S,\underline{Z})$ on $\h_\S$.
As shown in Table \ref{table:Coxeter number}, the Coxeter numbers of $Q$ and $\S$ coincide.
On the other hand, the Auslander-Reiten translations of $\h_Q$ and $\h_\S$ also coincide under the equivalence $\h^F_Q\cong \h_\S$.
Then a direct calculation gives the list of $(\h_\S,\underline{Z})$ above, where the index of simple modules in $\h_\S$ is induced from the index of simple modules in $\h_Q$. Further, $(\h_\S,\underline{Z})$ is naturally a stability condition of Gepner type $(\tau,-\frac{2}{h})$, and the uniqueness of $(\h_Q,Z)$ guarantees the uniqueness of $(\h_\S,\underline{Z})$.
\end{proof}

To find all the non-trivial Gepner type stability conditions on $\D(Q)$ of simply-laced Dynkin type (excepting $D_4$ type) up to $\mathbb{C}$-action, by equalities \eqref{eq:auto group} and $\tau^h=[-2]$, we should consider the auto-equivalences $\tau^m, 0 < m < h,$ and $\tau^mF, 0\leq m < h$. For the non-simply-laced case, we only consider the auto-equivalence $\tau^m, 0 < m < h$ by \eqref{eq:auto group2}. Assume $\sigma$ is a stability condition on $\D(Q)$ of Gepner type $(\tau^m, s)$, since $\tau^h=[-2]$, $s$ must be $-\frac{2m}{h}$. When $m$ and $h$ are coprime, it is not hard to see that $\sigma$ is also of Gepner type $(\tau, -\frac{2}{h})$. Otherwise, there are other (and in fact infinite many) stability conditions of Gepner type $(\tau^m, s)$ up to $\mathbb{C}$-action, see the example of $A_5$ type as follows.
\begin{example}\label{ex:Gepner}
We consider $\D(Q)$ of $A_5$ type with $Q$ as in Figure \ref{fig:orientation}. The Auslander-Reiten quiver
of $\h_Q$ is as follows.
\[
\xymatrix@R=.3pc@C=.8pc{
    T_5\ar[dr] && U_5\ar[dr] && V_5\ar[dr]\\
    & U_4 \ar[ur]\ar[dr] && V_4\ar[ur]\ar[dr] &&T_4 \\
    T_3\ar[ur]\ar[dr] && U_3\ar[dr]\ar[ur] && V_3\ar[dr]\ar[ur]\\
    & U_2 \ar[ur]\ar[dr] && V_2\ar[ur]\ar[dr] &&T_2 \\
    T_1\ar[ur] && U_1\ar[ur] && V_1\ar[ur]}
\]
Figure \ref{fig:Gepner}
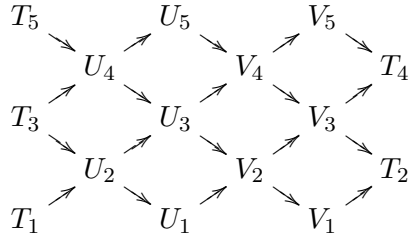
shows the central charge $Z$ of stability condition $\sigma$ of Gepner type $(\tau,-\frac{2}{h})$, which is given in Proposition \ref{prop:Gepner}. Note that $\sigma$ is also of Gepner type $(\tau^2,-\frac{4}{h})$. Thanks to the mesh relations and the Gepner type relation $Z(\tau(E))=e^{2\mathrm{i}\pi /h}\cdot Z(E)$, $Z$ is determined by the image of one indecomposable object in $\D(Q)$. However, for the case of Gepner type $(\tau^2,-\frac{4}{h})$, we need at least two images. So to construct another stability condition of Gepner type $(\tau^2,-\frac{4}{h})$, for example, let $Z'(T_1)=\frac{6}{5} Z(T_1)$ and $Z'(T_2)=Z(T_2)$. Then $Z'$ depicted in Figure \ref{fig:Gepner2} gives a stability condition $\sigma'$ of Gepner type $(\tau^2,-\frac{4}{h})$ with heart $\tilt{\h_Q}{\flat}{T_4}$.
\begin{figure}[ht]\centering
\begin{tikzpicture}[scale=4]
\draw[-] (-1.5,0) -- (0,0) node[below] {$0$};
\draw[dashed,->,>=latex] (0,0) -- (1.5,0) node[right]{$x$};

\path (0,0.2886751346*2) coordinate (U5);
\draw[thick,->,>=latex] (0,0) -- (U5) node[above]{\bahao $~Z(U_5)~ Z(U_1)$};

\path (0.5,0.2886751346) coordinate (T1);
\draw[thick,->,>=latex] (0,0) -- (T1) node[right]{\bahao $Z(T_1),Z(T_5)$};

\path (-0.5,0.2886751346) coordinate (V5);
\draw[thick,->,>=latex] (0,0) -- (V5) node[left]{\bahao $Z(V_5),Z(V_1)$};

\path (-1,0) coordinate (T2);
\draw[thick,->,>=latex] (0,0) -- (T2) node[below]{{\qihao 1}} node[above]{\bahao $Z(T_2),Z(T_4)$};

\path (-0.5,0.2886751346*3) coordinate (V4);
\draw[thick,->,>=latex] (0,0) -- (V4) node[left]{\bahao $Z(V_4),Z(V_2)$};

\path (0.5,0.2886751346*3) coordinate (U4);
\draw[thick,->,>=latex] (0,0) -- (U4) node[right]{\bahao $Z(U_4),Z(U_2)$};

\path (-1,0.2886751346*2) coordinate (V3);
\draw[thick,->,>=latex] (0,0) -- (V3) node[left]{\bahao $Z(V_3$)};

\path (1,0.2886751346*2) coordinate (T3);
\draw[thick,->,>=latex] (0,0) -- (T3) node[right]{\bahao $Z(T_{3}$)};

\path (0,0.2886751346*4) coordinate (U3);
\draw[thick,->,>=latex] (0,0) -- (U3) node[right]{\bahao $Z(U_3$)};

\draw [dotted] (0,0)--(1.5,0.2886751346*3)node[above]{$\frac{\pi}{6}$};
\draw [dotted] (0,0)--(0.5*1.45,0.2886751346*3*1.45)node[above]{$\frac{2\pi}{6}$};
\draw [dotted] (0,0)--(0,1.5)node[above]{$\frac{3\pi}{6}$};
\draw [dotted] (0,0)--(-0.5*1.45,0.2886751346*3*1.45)node[above]{$\frac{4\pi}{6}$};
\draw [dotted] (0,0)--(-1.5,0.2886751346*3)node[above]{$\frac{5\pi}{6}$};
\draw [dotted] (0.5,0.2886751346)--(0.5,0) node[below]{\qihao $\frac{\sqrt{3}}{6}$};
\end{tikzpicture}
\caption{A central charge of Gepner type $(\tau,-\frac{2}{h})$}\label{fig:Gepner}
\end{figure}

\begin{figure}[ht]\centering
\begin{tikzpicture}[scale=4]
\draw[-] (-1.5,0) -- (0,0) node[below] {$0$};
\draw[dashed,->,>=latex] (0,0) -- (1.5,0) node[right]{$x$};

\path (0.6,0.2886751346*1.2) coordinate (T1);
\draw[thick,->,>=latex] (0,0) -- (T1) node[right]{\bahao $Z'(T_1)$};

\path (-1.1,-0.2886751346*0.6) coordinate (T4);
\draw[dashed] (0,0) -- (T4);

\path (1.1,0.2886751346*0.6) coordinate (T4[1]);
\draw[thick,->,>=latex] (0,0) -- (T4[1]) node[below]{\bahao $Z'(T_{4}[-1])$};

\path (-0.6,0.2886751346*1.2) coordinate (V1);
\draw[thick,->,>=latex] (0,0) -- (V1) node[above]{\bahao $Z'(V_1)$};

\path (-0.1,0.2886751346*4.2) coordinate (U3);
\draw[thick,->,>=latex] (0,0) -- (U3) node[above]{\bahao $Z'(U_3)$};

\path (0.4,0.2886751346*1.2) coordinate (T5);
\draw[thick,->,>=latex] (0,0) -- (T5) node[above]{\bahao $Z'(T_{5})$};

\path (0.4,0.2886751346*3.6) coordinate (U4);
\draw[thick,->,>=latex] (0,0) -- (U4) node[above]{\bahao $Z'(U_4)$};

\path (-0.1,0.2886751346*1.8) coordinate (U1);
\draw[thick,->,>=latex] (0,0) -- (U1) node[left]{\bahao $Z'(U_1)$};

\path (-0.7,0.2886751346*3) coordinate (V2);
\draw[thick,->,>=latex] (0,0) -- (V2) node[above]{\bahao $Z'(V_2)$};

\path (-1.1,0.2886751346*1.8) coordinate (V3);
\draw[thick,->,>=latex] (0,0) -- (V3) node[left]{\bahao $Z'(V_3)$};

\path (0,0.2886751346*2.4) coordinate (U5);
\draw[thick,->,>=latex] (0,0) -- (U5) node[right]{\bahao $~Z'(U_5)$};

\path (-1,0) coordinate (T2);
\draw[thick,->,>=latex] (0,0) -- (T2) node[below]{\qihao 1}node[above]{\bahao $Z'(T_2)$};

\path (-0.5,0.2886751346*3) coordinate (V4);
\draw[thick,->,>=latex] (0,0) -- (V4) node[right]{\bahao $Z'(V_4)$};

\path (-0.5,0.2886751346*0.6) coordinate (V5);
\draw[thick,->,>=latex] (0,0) -- (V5) node[below]{\bahao $Z'(V_5)$};

\path (0.5,0.2886751346*3) coordinate (U2);
\draw[thick,->,>=latex] (0,0) -- (U2) node[right]{\bahao $Z'(U_2)$};

\path (1,0.2886751346*2.4) coordinate (T3);
\draw[thick,->,>=latex] (0,0) -- (T3) node[above]{\bahao $Z'(T_{3}$)};

\draw [dotted] (0,0)--(1.5,0.2886751346*3)node[above]{$\frac{\pi}{6}$};
\draw [dotted] (0,0)--(0.5*1.45,0.2886751346*3*1.45)node[above]{$\frac{2\pi}{6}$};
\draw [dotted] (0,0)--(0,1.5)node[above]{$\frac{3\pi}{6}$};
\draw [dotted] (0,0)--(-0.5*1.45,0.2886751346*3*1.45)node[above]{$\frac{4\pi}{6}$};
\draw [dotted] (0,0)--(-1.5,0.2886751346*3)node[above]{$\frac{5\pi}{6}$};
\draw [dotted] (0.5*1.2,0.2886751346*1.2)--(0.5*1.2,0) node[below]{\qihao $\frac{\sqrt{3}}{5}$};
\end{tikzpicture}
\caption{A central charge of Gepner type $(\tau^2,-\frac{4}{h})$}\label{fig:Gepner2}
\end{figure}
\end{example}

Final remark is that Proposition \ref{prop:Gepner2} gives the minimal value of gldim for species of Dynkin type.
\begin{corollary}
\label{thm:Gepner}
Let $\S$ be a species of Dynkin type. The range of the global dimension $\gldim$ on $\Stab\D(\S)$ is $[1-\frac{2}{h}, +\infty)$. Moreover, $\underline{\sigma}_G$ is the unique minimal point, with value $1-\frac{2}{h}$, of $\gldim$.
\end{corollary}
\begin{proof}
The proof is similar to the proof of \cite[Theorem 4.6]{Q3}, noticing that $\Stab\D(\S)$ is connected and we have Proposition \ref{prop:Gepner2}.
\end{proof}


\end{document}